\documentclass[acmtocl, acmnow]{acmtrans2m-play-sansperms}

\usepackage{amsmath,amsfonts,amssymb,euscript, graphicx,epsfig,
enumerate,float,afterpage, subfigure, ifthen}%

\newtheorem{thm}{Theorem}

\newtheorem{lem}{Lemma}
\newtheorem{defn}{Definition}

\newcommand{\expect}[1]{\mathbb{E}\left\{#1\right\}}
\newcommand{\defequiv}{\mbox{\raisebox{-.3ex}{$\overset{\vartriangle}{=}$}}}

\newcommand{\bv}[1]{{\boldsymbol{#1} }}
\newcommand{\script}[1]{{{\cal{#1} }}}

\newcommand{\dmax}{{\delta}}

\title{Dynamic Data Compression with Distortion Constraints 
 for Wireless Transmission \\ over a Fading Channel}
\author{Michael J. Neely  ,  Abhishek Sharma}

\begin{abstract}
We consider a wireless node that randomly receives data from
different
sensor units.    The arriving data must be compressed, stored, and
transmitted over a wireless link, where both the compression and transmission
operations consume power.  Specifically, the controller must choose from
one of multiple compression options every timeslot. 
Each option requires a different amount of
power and has different
compression ratio properties.
Further, the wireless link has potentially
time-varying channels, and 
transmission rates
depend on current channel states and transmission power allocations.
We design a dynamic algorithm for joint compression
and transmission, and prove that it 
comes arbitrarily close to minimizing
average power expenditure,
with an explicit tradeoff in average delay.  
Our approach uses stochastic network optimization
together with  a concept of
\emph{place holder bits} to provide efficient  energy-delay performance. 
The algorithm is simple to implement and does 
not require knowledge of
probability distributions for packet arrivals or channel states.
Extensions that treat 
distortion constraints are also considered. 
\end{abstract}

\category{...}{...}{...}

\terms{queueing analysis, stochastic network optimization, sensor networks, 
data fusion, distortion, fading channel}

\begin{document} 

\begin{bottomstuff} 
This work was presented in part as an invited paper at the Conference on 
Information Sciences and Systems (CISS), 
Princeton, NJ, March 2008 \cite{dynamic-compression-ciss08}.

Michael J. Neely and Abhishek Sharma are with the  
Electrical Engineering Department and the Computer Science Department, respectively, 
at the University of Southern California, Los Angeles, CA (web: http://www-rcf.usc.edu/$\sim$mjneely).

This material is supported in part  by one or more of
the following: the DARPA IT-MANET program
grant W911NF-07-0028, the NSF grant OCE-0520324, the NSF Career grant CCF-0747525.
\end{bottomstuff}

\maketitle

\nocite{dynamic-compression-ciss08} 

\section{Introduction}

We consider the problem of energy-aware data compression and
transmission for a wireless link that receives data from $N$ different
sensor units (Fig. \ref{fig:data-fusion1}).  Time is slotted with normalized
slot durations $t \in \{0, 1, 2, \ldots\}$,
and every timeslot the link receives a packet from a random number of the sensors.
We assume that packets arriving on the same timeslot contain correlated data, and
that this data can be compressed
using one of multiple compression options.   However, the signal processing
required for compression consumes a significant amount of energy,
and more sophisticated compression algorithms  are also more energy expensive.
 Further, the data must be transmitted over a wireless channel with potentially
varying channel conditions, where the transmission rates available on the current
timeslot depend on the current channel condition and the current transmission power
allocation.   The goal is to design a joint compression and transmission scheduling policy
that minimizes time average power expenditure.

This problem is important for modern sensor networks where
correlated (and compressible) data flows over power limited nodes.
Compressing the data can save power by reducing the amount of bits
that need to be transmitted, provided that the transmission power
saved is more than the power expended in the compression operation.
It is important to understand the optimal balance between
compression power and transmission power.  Work in 
\cite{energy-aware-compression} considers this question for a 
wireless link with fixed transmission costs, and describes practical compression issues
and reports communication-to-computation energy ratios  for popular algorithms.
Work in \cite{sadler-sensys} considers a similar static situation
where  the wireless channel
condition is the same for all time. There, it is shown experimentally that
compression can lead to a significant power savings when data is
transmitted over multiple hops. The proposed algorithm of
\cite{sadler-sensys} uses a fixed data compression scheme, an
adaptation of the Lempel-Ziv-Welch (LZW) compression algorithm for
sensor networks.   Techniques for distributed compression 
using Slepian-Wolf coding theory are considered in \cite{julius-distributed-compression}
\cite{networked-slepian-wolf}. 
Models of spatial correlation between data of different sensors
are proposed in \cite{pattem-correlation-model-v-distance-for-compression}
and used to construct and evaluate 
energy-efficient routing algorithms that compress data at each stage.

The above prior work has concentrated on static environments 
where transmission power is directly proportional to the number of bits
transmitted and/or traffic rates are fixed and known, so that 
compression and transmission strategies
can be designed in advance.
Here, we focus attention on a single link, but consider a stochastic
environment where the amount of data received every slot
is random, as is the current channel
condition for wireless transmission.  
Further, the transmission rate 
is an arbitrary (possibly non-linear) function of transmission power. 
Optimal policies in this stochastic context are
more complex,  and more care is required to 
ensure transmissions are energy-efficient. 

\begin{figure}[top]
  \centering
   \includegraphics[height=2in, width=3in]{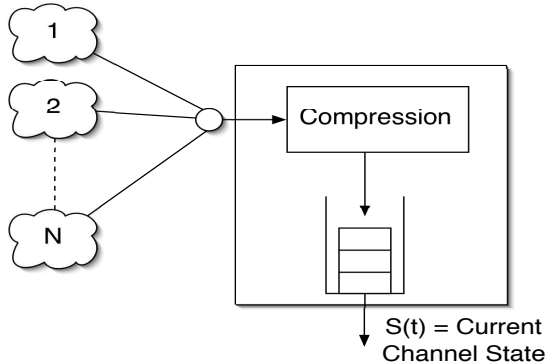} 
  \caption{Multiple sensors sending data to a single wireless link.}
  \label{fig:data-fusion1}
\end{figure}

In this paper, we design a dynamic compression and transmission scheduling
algorithm and prove that the algorithm pushes total time average power arbitrarily
close to optimal, with a corresponding tradeoff in average delay.  We assume the
algorithm has a table of expected compression ratios for each compression option,
and that, if channels are time-varying,  current channel state information is available.
Our algorithm bases decisions purely on this information and does not require a-priori
knowledge of the packet arrival or channel state probabilities.  The algorithm is
simple  to implement and is robust to situations where these probabilities can change.
 This work is important as
it demonstrates a principled method of making on-line compression decisions in a
stochastic system with correlated data.  Our solution applies the techniques of
Lyapunov optimization developed in our previous
work \cite{neely-energy-it} \cite{now}, and is perhaps the first application
of these techniques to the dynamic compression problem.  This paper also 
extends the general theory by introducing a novel concept of \emph{place-holder bits} 
to improve delay in stochastic networks with costs. 
Related Lyapunov optimization techniques for network flow control applications
are developed in \cite{neely-thesis} \cite{neely-fairness-infocom05}, 
and alternative fluid model approaches are developed in \cite{stolyar-greedy} \cite{atilla-fairness}.

In the next section we describe the system model, and in
Section \ref{section:optimal-definition} we characterize the minimum
average power in terms of an optimization problem based on channel and packet arrival
probabilities.  In Section \ref{section:algorithm} we develop an on-line algorithm
that makes simple decisions based only on current information.  The
algorithm achieves time average power that can be pushed arbitrarily close to optimum via a
simple control parameter that also affects an average delay tradeoff.  A simple  improvement 
via \emph{place-holder bits} is developed in Section \ref{section:improvement}.
Extensions to systems 
with distortion  constraints are given in Section \ref{section:distortion}.  Simulations are 
provided in Section \ref{section:simulation}.

\section{System Model} \label{section:formulation}

Consider the wireless link of Fig. \ref{fig:data-fusion1} that operates in slotted time and
receives packets from $N$
different sensor units.  If an individual sensor sends data during a timeslot,
this data is in the form of a fixed length packet of size $b$ bits, containing sensed information.
Let $A(t)$ represent the number of sensors that send packets during slot $t$, so that
$A(t) \in \{0, 1, \ldots, N\}$.  The data from these $A(t)$ packets may be correlated,
and hence it
may be possible to compress the information within the $A(t)$ packets
(consisting of $A(t)b$ bits)
into a smaller data unit for transmission
over the wireless link.  This is done via a \emph{compression function} $\Psi(a, k)$ defined
as follows.
There are $K+1$ compression options, comprising a set
$\script{K} = \{0, 1,\ldots, K\}$. Option $0$ represents no attempted
compression, and options $\{1, 2, \ldots, K\}$ represent
various alternative methods to compress the data.  The function
$\Psi(a,k)$ takes input $a \in \{0, 1, \ldots, N\}$ (representing the number of newly arriving
packets) and compression option $k \in \script{K}$, and generates a \emph{random
variable} output $R$, representing the total size of the data after compression.

 Every
timeslot the link controller observes the random number of new packet arrivals $A(t)$
and chooses a compression option $k(t) \in \script{K}$, yielding the random
compressed output $R(t) = \Psi(A(t), k(t))$.  Let $P_{comp}(t)$ represent the power expended
by this compression operation, and assume this is also a random function of the
number of packets compressed and the compression option.   We assume that the compressed
output  $R(t)$ 
is conditionally i.i.d. over all slots that have the same number of packet arrivals $A(t)$ and the same
compression decision $k(t)$.  Likewise, compression power $P_{comp}(t)$ is conditionally i.i.d. 
over all slots with the same $A(t)$ and $k(t)$. 
The average compressed
output $m(a,k)$ and the average power expenditure $\phi(a,k)$ associated with
$A(t) = a, k(t) = k$ are defined:
\begin{eqnarray}
m(a, k) &=& \expect{\Psi(A(t), k(t)) \left|\right. A(t) = a, k(t) = k} \label{eq:m-table} \\
\phi(a, k) &=& \expect{P_{comp}(t) \left|\right. A(t) = a, k(t) = k} \label{eq:phi-table}
\end{eqnarray}
We assume the values of  $m(a,k)$ and  $\phi(a,k)$ are known so that the following
table can be constructed:

\begin{figure}[cht]
\centering
\begin{tabular}{|c | c | c | c |}
\hline
 $k$&$\Psi(a,k)$&$\expect{\Psi(a,k)}$ &$\expect{P_{comp} \left|\right. a,  k}$ \\ \hline
 0  & $ab$&$ab$&$\phi(a,0) = 0$\\ \hline
  1  & Random & $m(a,1)$&$\phi(a, 1)$\\ \hline
   2 & Random & $m(a,2)$&$\phi(a, 2)$ \\ \hline
      $\cdots$ & $\cdots$& $\cdots$&$\cdots$ \\ \hline
   $K$& Random & $m(a,K)$&$\phi(a,K)$\\ \hline
 \end{tabular}
 \label{fig:table-fusion}
\end{figure}

Note that we assume $\Psi(a, 0) = ab$ and $\phi(a,0) = 0$, as the compression option
$k=0$ does not compress any data and also does not expend any power.
We further assume that $m(a,k) \leq ab$ for all $a \in \{0, 1, \ldots, N\}$ and
all $k\in \script{K}$, so that compression is not expected to expand
the data.

\subsection{Data Transmission and Queueing}

The compressed data $R(t) = \Psi(A(t), k(t))$ is delivered to a queueing buffer for transmission over the
wireless link (see Fig. \ref{fig:data-fusion1}). Let $U(t)$ represent the current
number of bits (or \emph{unfinished work})  in the queue. The queue
backlog evolution is given by:
\begin{equation} \label{eq:queue-dynamics}
U(t+1) = \max[U(t) - \mu(t), 0] + R(t)
\end{equation}
where $\mu(t)$ is the transmission rate offered by the link on slot $t$. This rate
is determined by the current channel condition and the current transmission power allocation
decision, as in \cite{now}.
Specifically, the channel is assumed to be constant over the duration of a slot, but
can potentially change from slot to slot. Let $S(t)$ represent the
current channel state, which is assumed to take values in some finite set $\script{S}$.
We assume the channel state $S(t)$ is known at the beginning of each slot $t$, so that
the link can make an opportunistic transmission power allocation decision $P_{tran}(t)$,
yielding a transmission rate $\mu(t)$ given by:
\[ \mu(t) = C(P_{tran}(t), S(t)) \]
where $C(P, s)$ is the \emph{rate-power} curve associated with the modulation and
coding schemes used  for transmission over the channel.  We assume $C(P, s)$ is 
continuous in power $P$ for each channel state $s \in \script{S}$. Transmission  
power allocations $P(t)$ are restricted to some compact set $\script{P}$ for all slots $t$, 
where $\script{P}$ contains a
maximum transmission power $P_{max}$.   For example, the set $\script{P}$ can contain a
discrete set of power levels, such as the
two element set
$\script{P} = \{0, P_{max}\}$.  Alternatively, $\script{P}$  can be a continuous interval,
such as $\script{P} = \left\{ P \left|\right. 0 \leq P \leq P_{max}\right\}$.  We assume throughout
that $0 \in \script{P}$
and that $C(0, s) = 0$ for all channel states $s \in \script{S}$, so that zero
transmission power yields a zero transmission rate.  Further, we assume that
$C(P_{max}, s) \geq C(P, s)$ for all $s \in \script{S}$ and all $P \in \script{P}$, so that
allocating maximum power yields the largest transmission rate that is
possible under the given channel state.


\subsection{Stochastic Assumptions and the Control Objective}

For simplicity, we assume the packet arrival process $A(t)$ is i.i.d. over slots with
a general probability distribution $p_A(a) = Pr[A(t) = a]$.  Likewise, the
channel state process $S(t)$ is i.i.d.
over slots with a general distribution $\pi_s = Pr[S(t) = s]$.\footnote{Using the $T$-slot
Lyapunov drift techniques described in \cite{now}, our analysis
can be generalized to show that the \emph{same} algorithms we derive under the i.i.d.
assumption yield  similar performance for arbitrary ergodic arrival and channel processes
$A(t)$ and  $S(t)$, with
delay bounds that increase by  a constant factor related to the mixing times of the processes.}
The distributions $p_A(a)$ and $\pi_s$ are not necessarily known to the link controller.
Every slot the link controller observes the number of new packets $A(t)$, the current
queue backlog $U(t)$,
and the current channel state $S(t)$, and makes a compression decision $k(t) \in \script{K}$
(expending power $P_{comp}(t)$)
and a transmission power allocation $P_{tran}(t) \in \script{P}$.  The total time average power expenditure
is given by:
\[ \lim_{t\rightarrow\infty} \frac{1}{t} \sum_{\tau=0}^{t-1} [P_{comp}(\tau) + P_{tran}(\tau)] \]
The goal is to make compression and transmission decisions to minimize time average power
while ensuring the queue $U(t)$ is stable.  Formally, we define a queueing process
$U(t)$ to be stable if:
\[ \limsup_{t\rightarrow\infty} \frac{1}{t}\sum_{\tau=0}^{t-1} \expect{U(\tau)} < \infty \]
This type of stability is often referred to as \emph{strong stability}, as it implies a finite average
backlog and hence a finite average delay.  In Section \ref{section:algorithm}, we shall design a class
of dynamic
algorithms that can drive time average power arbitrarily close to the minimum average power
required for stability, with a corresponding explicit tradeoff in average queue backlog and average
delay.

Define $r_{min}$ and $r_{max}$ as follows:
\begin{eqnarray}
 r_{min} &\defequiv& \expect{\min_{k \in \script{K}} m(A(t), k)} \label{eq:rmin} \\
 r_{max} &\defequiv& \expect{C(P_{max}, S(t))} \label{eq:rmax}
 \end{eqnarray}
 where the expectations are taken over the randomness of $A(t)$ and $S(t)$ via the
 distributions $p_A(a)$ and $\pi_s$.
Thus, $r_{min}$ is the minimum average bit rate delivered to the queueing system (in units
of bits/slot), assuming
the compression option that results in the largest expected bit reduction is used every slot.
The value
$r_{max}$ represents the maximum possible average transmission rate over the wireless link.
We assume throughout
that $r_{min} < r_{max}$, so that it is possible to stabilize the system.

Thus, there are two reasons to compress data:  (i) In order to
stabilize the queue, we may need to compress (particularly if $\expect{A(t)}b > r_{max}$).
(ii) We may actually save power if the power used to compress is less than the extra amount of power that
would be used transmitting the extra data if it were not compressed.

\subsection{Discussion of the System Model}

This simple model captures a wide class of systems where data compression
is important.  The $N$ sensor scenario of Fig. \ref{fig:data-fusion1} captures
the possibility of randomly arriving data that is \emph{spatially correlated}.
An example is when
there are multiple sensors in an environment and only a random subset of them
detect a particular event.  The data provided by these sensors
is thus correlated but not necessarily identical, as each observation
can offer new information.

The case of compression because of
\emph{time correlated} data can also be treated in this model by re-defining $N$ to
represent the time over which a frame of data samples are gathered.  Indeed, suppose
a timeslot $t$ is composed of $N$ mini-slots, where data can arrive on any or all of
the mini-slots.  The value of $A(t)$ now represents the  random number
of packets arriving over the $N$ mini-slots, and the compression
functions $m(a,k)$ and $\phi(a,k)$ now represent averages associated with compressing
the time-correlated data.  This of course assumes compression is contained to data
arriving within the same frame, and does not treat inter-frame compression.

Our time-varying channel model is useful for systems with mobility, environmental
changes, or restrictions that create time-varying transmission opportunities. 
This allows for opportunistic scheduling which can help to further reduce power
expenditure. We do not consider the additional power required to measure the
channel conditions here.  Extensions that 
treat this issue can likely be obtained using the techniques for optimizing
measurement decisions developed in 
\cite{chih-ping-channel-measure}.
A special case of the time-varying channel model is the \emph{static channel} assumption,
where $S(t)$ is the same for all timeslots $t$.  This special case is similar to the static assumption
in \cite{energy-aware-compression} \cite{sadler-sensys}. 
However, this static channel scenario still creates an 
interesting problem that is much
different from \cite{energy-aware-compression} \cite{sadler-sensys}.  Indeed, 
the random packet arrivals (with raw data rate that is possibly larger than link capacity) 
and the potentially non-linear rate-power curve
necessitate a 
dynamic compression
strategy that is not obvious, that depends on the packet arrival distribution,  
and that does not necessarily use the same compression 
option on every slot.

Here we assume that the compression options available within the set $\script{K}$
are sufficient to
ensure that the resulting data transmitted over the link has an acceptable
fidelity.  An example is \emph{lossless} data compression, such as Huffman
or Lempel-Ziv source coding, where all original data packets can be reconstructed
at the destination.  Alternatively, we might have some compression options
$k \in \script{K}$ representing \emph{lossy} compression, provided that the
distortion that may be introduced is acceptable.   Extensions to systems that 
explicitly consider distortion due to lossy compression 
are considered in Section \ref{section:distortion}.

\section{Minimum Average Power} \label{section:optimal-definition}

Here we characterize the minimum time average power
required for queue stability.  We first define separate
functions $h^*(r)$ and $g^*(r)$ that describe the minimum average
power for compression and transmission, respectively, over a restricted
class of stationary randomized algorithms.  These functions depend on the
steady state arrival and channel distributions $p_A(a)$ and $\pi_s$.
We then show that these
functions can be used to define system optimality
over the class of all possible decision strategies, including strategies that
do not necessarily make stationary and randomized decisions.

\subsection{The  Functions $h^*(r)$ and $g^*(r)$}
\begin{defn} \label{def:1}  For any value $r$ such that $r_{min} \leq r \leq b\expect{A(t)}$,
the \emph{minimum-power compression function} $h^*(r)$ is defined
as the \emph{infimum value} $h$ for which there exist probabilities $(\gamma_{a,k})$
for $a \in \{0, 1, \ldots, N\}$, $k \in \script{K}$,
such that the following constraints are satisfied:
\begin{eqnarray}
\sum_{a=0}^N \sum_{k=1}^{K} p_A(a) \gamma_{a,k} \phi(a, k) = h \label{eq:statphi} \\
\sum_{a=0}^N \sum_{k=1}^K p_A(a) \gamma_{a,k} m(a,k) \leq  r \label{eq:statm} \\
\gamma_{a,k} \geq 0  \: \:   \mbox{ for all $a , k$} \label{eq:gamma1} \\
\sum_{k=1}^K \gamma_{a,k} = 1 \: \:   \mbox{ for all $a$}  \label{eq:gamma2}
\end{eqnarray}
\end{defn}

Intuitively, the $(\gamma_{a,k})$ values define a stationary randomized policy that
observes the current arrivals $A(t)$ and uses compression option $k$ with probability
$\gamma_{a,k}$ whenever $A(t)=a$.  The expression on the left hand side of
(\ref{eq:statphi}) is the expected compression
power $\expect{P_{comp}(t)}$ for this policy.  Likewise, the expression on the left
hand side of (\ref{eq:statm}) is the expected number of bits $\expect{R(t)}$ at the output
of the compressor for this policy.  The value of $h^*(r)$ is thus the smallest possible
average power due to compression, infimized over all such stationary randomized
policies that yield  $\expect{R(t)} \leq r$.  Note from
(\ref{eq:rmin}) that it
is possible to have a stationary randomized policy that yields $\expect{R(t)} = r_{min}$,
and hence the function $h^*(r)$ is well defined for
any $r \geq r_{min}$.  Further, the following lemma shows that 
the infimum value $h^*(r)$ 
can be \emph{achieved} by a particular stationary randomized algorithm.

\begin{lem} \label{lem:exist-particular}
For any $r$ such that $r_{min} \leq r \leq b\expect{A(t)}$, there exists a particular
stationary randomized policy that makes compression decisions $k^*(t)$ as a random function of the
observed $A(t)$ value (and independent of queue backlog), such that:
\begin{eqnarray}
 \expect{\phi(A(t), k^*(t))} &=& h^*(r)  \label{eq:stat-h}\\
\expect{m(A(t), k^*(t))} &=& r \label{eq:stat-r}
\end{eqnarray}
where  the above expectations are taken with respect to the steady state packet arrival
distribution $p_A(a)$ and the randomized compression decisions $k^*(t)$.
\end{lem}
\begin{proof}
The proof follows by continuity of the functions on the left hand side of (\ref{eq:statphi})
and (\ref{eq:statm}) with respect to $\gamma_{a,k}$, and by compactness of the 
set of all $(\gamma_{a,k})$ that satisfy  (\ref{eq:gamma1}) and (\ref{eq:gamma2}). 
See Appendix A for details.  
\end{proof}

Similar to the function $h^*(r)$, we define $g^*(r)$ as the smallest possible average transmission
power required for a stationary randomized algorithm
to support a transmission rate of at least $r$.  The precise definition is given below.

\begin{defn}  \label{def:2} For any value $r$ such that $0 \leq r \leq r_{max}$, the
\emph{minimum-power transmission function} $g^*(r)$ is defined as the \emph{infimum
value} $g$ for which there exists a stationary randomized power allocation policy that
chooses transmission power $P_{tran}(t)$ as a random function of the observed channel state $S(t)$
(and independent of current queue backlog), such that:
\begin{eqnarray}
\expect{P_{tran}(t)} &=& g  \label{eq:g1} \\
\expect{C(P_{tran}(t), S(t))} &\geq&  r \label{eq:g2}
\end{eqnarray}
\end{defn}

The function $g^*(r)$ is well defined whenever $r \leq r_{max}$ because
it is possible to satisfy the constraint (\ref{eq:g2}).
Indeed, note by (\ref{eq:rmax}) that the policy $P_{tran}(t) = P_{max}$ for all $t$ yields
$\expect{C(P_{tran}(t), S(t))} = r_{max}$.    Furthermore, it is easy to show that
the inequality constraint in (\ref{eq:g2}) can be replaced by an equality constraint,
as any  policy with an average transmission rate larger than $r$ can be modified
to achieve rate $r$ exactly while using strictly less power.  This can be done by
independently setting $P_{tran}(t)=0$ with some probability every slot, yielding
a zero transmission rate in that slot.

Because the set $\script{P}$ is compact and the 
function $C(P,s)$ is continuous in power $P$ for all 
channel states $s \in \script{S}$, 
an argument similar to the proof of Lemma \ref{lem:exist-particular} can be used to show
that  the infimum average power $g^*(r)$ can be \emph{achieved}
by a particular stationary randomized policy.\footnote{More generally,
the infimum can be achieved
whenever
$C(P, s)$ is \emph{upper semi-continuous} in $P$ for every channel 
state $s \in \script{S}$ (see
\cite{bertsekas-convex} for a definition).
The upper semi-continuity property is a mild property that is true of every practical curve
$C(P,s)$. All results of this paper hold when continuity is replaced by upper 
semi-continuity.}  Specifically, for any $r$ such that $0 \leq r \leq r_{max}$, 
there exists a stationary randomized algorithm that chooses transmission power
$P_{tran}^*(t)$ that yields:
\begin{eqnarray}
\expect{C(P_{tran}^*(t), S(t))} &=& r \label{eq:stat-tran-r} \\
\expect{P_{tran}^*(t)} &=& g^*(r) \label{eq:stat-tran-g}
\end{eqnarray}

\subsection{Structural Properties of $h^*(r)$ and $g^*(r)$}
It is not difficult to show that  $h^*(r)$ is a non-increasing function of $r$ (because
less compression power is required if a larger compressor output rate is allowed), and
that $g^*(r)$ is a non-decreasing  function of $r$ (because more transmission power
is required to support a larger transmission rate). Further, both functions are convex.
It is interesting to note that in
the special case when there is
no channel state variation so that $C(P, s) = C(P)$, and when the function $C(P)$ is
strictly increasing and concave, then $g^*(r) = C^{-1}(r)$, i.e., it is given by the inverse
of $C(P)$.  Details on the structure of the $g^*(r)$ function in the general time-varying
case are given in \cite{neely-energy-it}.

\subsection{Minimum Average Power for Stability}

The following theorem establishes the minimum time average power required for
queue stability in terms of the $h^*(r)$ and $g^*(r)$ functions.  We consider all
possible algorithms for making compression decisions $k(t) \in \script{K}$ and
transmission power decisions $P_{tran}(t) \in \script{P}$ over time, including algorithms
that are not necessarily in the class of stationary randomized policies.

\begin{thm} \label{thm:min-power}  Let $A(t)$ and $S(t)$ be ergodic with steady
state distributions $p_A(a)$ and $\pi_s$, respectively (such as processes that
are i.i.d. over slots, or more general Markov modulated processes).  Assume
that $r_{min} < r_{max}$ (defined in (\ref{eq:rmin}), (\ref{eq:rmax})). Then any joint compression
and transmission rate scheduling algorithm that stabilizes the queue $U(t)$ yields a
time average power expenditure that satisfies:
\begin{eqnarray*}
\limsup_{t\rightarrow\infty} \frac{1}{t} \sum_{\tau=0}^{t-1} \expect{P_{comp}(\tau) + P_{tran}(\tau)} \geq P_{av}^*
\end{eqnarray*}
where $P_{av}^*$ is defined as the optimal solution to the following problem:
\begin{eqnarray}
\mbox{Minimize:} &  h^*(r) + g^*(r)   \label{eq:min-h-and-g}\\
\mbox{Subject to:} & r_{min} \leq r \leq \min[r_{max}, b\expect{A(t)}] \label{eq:bounds-h-and-g}
\end{eqnarray}
\end{thm}
\begin{proof}
See Appendix B.
\end{proof}

The above theorem shows that time average power must be greater than or equal
to $P_{av}^*$ for  queue stability.  The result can be understood intuitively by observing
that if $r$ is the rate of bits arriving to the queue from the compressor, then
average transmission power can  be minimized while maintaining stability
by pushing the time average transmission rate down closer and
closer to $r$.   The optimization problem corresponding to this
definition of $P_{av}^*$ may be difficult to solve in practice, as it would require
exact knowledge of the $h^*(r)$ and $g^*(r)$ functions, which in turn requires
full a-priori knowledge  of the distributions $p_A(a)$ and  $\pi_s$.
 In the next section, we design a simple class of dynamic
algorithms that stabilize the queue without this knowledge, and that push time
average power arbitrarily close to $P_{av}^*$.

\section{The Dynamic Compression Algorithm} \label{section:algorithm}

Our dynamic algorithm is decoupled into separate policies for data compression and
transmission rate scheduling.  It is defined in terms of a control parameter $V>0$ that
affects an energy-delay tradeoff.

\emph{The Dynamic Compression and Transmission Algorithm}:

\emph{\underline{Compression}:}  Every slot $t$, observe the number of new packet arrivals
$A(t)$ and the current queue backlog $U(t)$,
and choose compression option $k(t) \in \script{K}$ as follows:
\begin{equation} \label{eq:compress-alg1} 
k(t) = \arg \min_{k \in \script{K}} \left[U(t)m(A(t), k) + V\phi(A(t), k) \right] 
\end{equation}
If there are multiple compression options $k \in \script{K}$ that minimize
$U(t) m(A(t), k)  + V\phi(A(t), k)$, break ties arbitrarily.

\emph{\underline{Transmission}:}   Every slot $t$, observe the current channel
state $S(t)$ and the current queue backlog $U(t)$, and choose transmission
power $P_{tran}(t) \in \script{P}$ as follows:
\begin{equation} \label{eq:transmit-alg1} 
P_{tran}(t) = \arg \max_{P \in \script{P}} \left[U(t)C(P, S(t)) - VP  \right] 
\end{equation} 
Recall that $\script{P}$ is assumed to be compact and the $C(P, s)$ function
is upper semi-continuous, and hence there exists
a maximizing power allocation.  If there are multiple power options that maximize
$U(t)C(P, S(t)) - VP$, break ties arbitrarily.

The compression policy involves a simple comparison of $K+1$ values
found by evaluating the
$m(a,k)$ and $\phi(a,k)$ functions for all $k \in \script{K}$,
and  can easily be accomplished
in real time.   The transmission policy is a special
case of the  Energy Efficient Control Algorithm (EECA)  policy developed
in  \cite{neely-energy-it}, and typically
can also be solved  very simply in real time.  The next theorem establishes
the performance of the combined algorithm.

\begin{thm} \label{thm:performance} (Algorithm Performance)
Suppose packet arrivals $A(t)$ are i.i.d. over slots with
distribution $p_A(a)$, and
channel states $S(t)$ and  are i.i.d. over slots with distribution $\pi_s$.
For any control parameter $V>0$, the dynamic compression and transmission
scheduling algorithm yields power expenditure and queue backlog that satisfy
the following:
\begin{eqnarray}
\overline{P}_{tot}  &\leq&
P_{av}^* + B/V \label{eq:power-bound} \\
\overline{U} &\leq& \frac{B + V(P_{max} + \phi_{max})}{(r_{max} - r_{min})} \label{eq:u-bound}
\end{eqnarray}
where $\overline{P}_{tot}$ and $\overline{U}$ are the time averages for power expenditure and
queue backlog, defined:
\begin{eqnarray*}
\overline{P}_{tot} &\defequiv& \limsup_{t\rightarrow\infty} \frac{1}{t} \sum_{\tau=0}^{t-1} \expect{P_{comp}(\tau) + P_{tran}(\tau)} \\
\overline{U} &\defequiv& \limsup_{t\rightarrow\infty} \frac{1}{t} \sum_{\tau=0}^{t-1} \expect{U(\tau)}
\end{eqnarray*}
and where $B$ and $\phi_{max}$ are constants given by:
\begin{eqnarray}
B &\defequiv& \frac{1}{2}\left[\sigma^2 + \expect{C(P_{max}, S(t))^2} \right] \label{eq:B} \\
\phi_{max} &\defequiv& \expect{\max_{k \in \script{K}}\left[\phi(A(t), k)\right]} \label{eq:phi-max}
\end{eqnarray}
where $\sigma^2$ is an upper bound on $\expect{R(t)^2}$ for all slots $t$.  For example, 
if no compression operation expands the data, then 
$R(t) = \Psi(A(t),k) \leq bA(t)$ for all $t$, and hence $\sigma^2$ is defined: 
\[ \sigma^2 \defequiv b^2\expect{A(t)^2} \]
\end{thm}

We prove Theorem \ref{thm:performance} in the next subsection.  Note that the parameter $V>0$
can be chosen to make $B/V$ arbitrarily small, ensuring by (\ref{eq:power-bound}) that time
average power is arbitrarily close to the optimal value $P_{av}^*$.  However, the resulting
average queue backlog bound
grows linearly with $V$.  By Little's Theorem, the average queue backlog is proportional
to average delay \cite{bertsekas-data-nets}.  This establishes an explicit tradeoff between
average power expenditure 
and delay.

As an implementation detail, we note for simplicity that we can use 
 units of bits, bits/slot, and 
 milli-Watts for $U(t)$, $C(P, S)$, and $P$.  However, these units are arbitrary and
 any consistent units will work, with performance given by (\ref{eq:power-bound}) and 
 (\ref{eq:u-bound}).   
 Indeed, any unit changes are captured in the $V$ constant (where $V$ has units of
 $\mbox{bits}^2/mW$ for the units above).  For example,
 if milli-Watts are changed to Watts, then the algorithm will make the exact same control
 decisions for $k(t)$ and $P_{tran}(t)$ over time, and hence yields the exact same sample
 path of energy use and queue backlog, 
 as long as the $V$ constant is appropriately changed by a factor of $1000$. If 
 bits are changed to kilobits, then $V$ must change by a factor of $10^6$.

\subsection{Lyapunov Performance Analysis for Theorem \ref{thm:performance}}

Our proof of Theorem \ref{thm:performance} 
relies on the performance optimal Lyapunov scheduling techniques
from \cite{now}  \cite{neely-energy-it}.  First define the following
quadratic Lyapunov function of queue backlog $U(t)$:
\[ L(U(t)) \defequiv \frac{1}{2}U(t)^2 \]
Define the one-step conditional Lyapunov drift $\Delta(U(t))$ as
follows:\footnote{More complete
notation would be $\Delta(U(t), t)$, as the drift depends on the scheduling policy which may
also depend on time $t$.  However,  we use the simpler notation $\Delta(U(t))$ as a formal
representation of the right hand side of (\ref{eq:drift-def}).  See \cite{now} for further details
on Lyapunov drift.}
\begin{equation} \label{eq:drift-def}
\Delta(U(t)) \defequiv \expect{L(U(t+1)) - L(U(t))\left|\right. U(t)}
\end{equation}
The following simple lemma from \cite{now} shall be useful.
\begin{lem} \label{lem:lyap-drift}  (Lyapunov drift \cite{now})
Let $L(U(t))$ be a non-negative function of $U(t)$ with
Lyapunov drift $\Delta(U(t))$ defined in (\ref{eq:drift-def}).  If
there are stochastic processes $\alpha(t)$ and $\beta(t)$
such that every slot $t$ and for all possible values of $U(t)$, the conditional Lyapunov drift
satisfies:
\begin{equation}  \label{eq:lyap-cond}
\Delta(U(t)) \leq \expect{\beta(t) - \alpha(t) \left|\right. U(t)}
\end{equation}
then:
\begin{eqnarray*}
\limsup_{t\rightarrow\infty} \frac{1}{t} \sum_{\tau=0}^{t-1} \expect{\alpha(\tau)} \leq \limsup_{t\rightarrow\infty}
\frac{1}{t} \sum_{\tau=0}^{t-1} \expect{\beta(\tau)}
\end{eqnarray*}
\end{lem}

The proof involves taking expectations of (\ref{eq:lyap-cond}), using iterated expectations, and
summing the resulting telescoping series (see \cite{now} for details).

The queue backlog $U(t)$ for our system satisfies the queue evolution equation (\ref{eq:queue-dynamics}).
Specifically, the queue has arrival process $R(t) = \Psi(A(t), k(t))$ and transmission rate process $\mu(t) = C(P_{tran}(t), S(t))$, where the $k(t)$ and $P_{tran}(t)$ control decisions are
determined by the dynamic compression and transmission algorithm of the previous
sub-section.  The Lyapunov drift is given by the following lemma.

\begin{lem} \label{lem:compute-drift} (Computing $\Delta(U(t))$) Under the queue evolution
equation (\ref{eq:queue-dynamics}) and using the quadratic
Lyapunov function $L(U(t)) = \frac{1}{2} U(t)^2$,  the Lyapunov drift $\Delta(U(t))$ satisfies the
following for all $t$ and all $U(t)$:
\begin{eqnarray}
\Delta(U(t)) \leq B - U(t) \expect{\mu(t) - m(A(t), k(t)) \left|\right. U(t)}  \label{eq:compute-drift}
\end{eqnarray}
where $\mu(t) = C(P_{tran}(t), S(t))$, and $B$ is given in (\ref{eq:B}). The expectation above
is taken with respect to the random channels and arrivals $S(t)$ and   $A(t)$, and the potentially
random control actions $k(t)$ and $P_{tran}(t)$.
\end{lem}

\begin{proof}
From (\ref{eq:queue-dynamics}) we have:
\begin{eqnarray*}
\frac{1}{2} U(t+1)^2 &=& \frac{1}{2}\left( \max[U(t) - \mu(t) , 0] + R(t)\right)^2 \\
&\leq& \frac{1}{2}\left[U(t)^2 + \mu(t)^2 + R(t)^2\right]  \\
&& - U(t)(\mu(t) - R(t))
\end{eqnarray*}
and hence (taking conditional expectations given $U(t)$):
\begin{eqnarray*}
\Delta(U(t)) &\leq& \frac{1}{2}\expect{\mu(t)^2  + R(t)^2 \left|\right. U(t)} \\
 && - U(t)\expect{\mu(t) - R(t)\left|\right.U(t)}
 \end{eqnarray*}
 It is clear that the value $\frac{1}{2} \expect{\mu(t)^2 + R(t)^2\left|\right. U(t)}$ is less than or
 equal to the constant  $B$ defined in (\ref{eq:B}), and hence:
 \begin{eqnarray}
\Delta(U(t)) \leq B - U(t) \expect{\mu(t) - R(t) \left|\right. U(t)} \label{eq:compute-drift0}
\end{eqnarray}
Noting that $R(t) = \Psi(A(t), k(t))$ and using iterated expectations, we have:
\begin{eqnarray*}
&& \hspace{-.5in} \expect{R(t) \left|\right.U(t)} \\
&=& \expect{\Psi(A(t), k(t)) \left|\right.U(t)} \\
&=& \expect{\expect{\Psi(A(t), k(t)) \left|\right.U(t), A(t), k(t)}\left|\right.U(t)} \\
&=& \expect{m(A(t), k(t))\left|\right.U(t)}
\end{eqnarray*}
where we have used the definition of $m(a,k)$ given in (\ref{eq:m-table}). Using this
equality in (\ref{eq:compute-drift0}) yields the result.
\end{proof}

Following the Lyapunov optimization framework of \cite{now} \cite{neely-energy-it}, 
we add a weighted cost term to the drift expression.  Specifically, 
from (\ref{eq:compute-drift}) we have:
\begin{eqnarray}
&& \hspace{-.2in} \Delta(U(t)) + V\expect{P_{comp}(t) + P_{tran}(t)\left|\right.U(t)} \leq  \nonumber \\
&& B - U(t) \expect{C(P_{tran}(t), S(t)) - m(A(t), k(t))\left|\right.U(t)} \nonumber \\
&& + V\expect{P_{comp}(t) + P_{tran}(t)\left|\right.U(t)} \label{eq:final-drift-metric}
\end{eqnarray}
where we have just added an additional term to both sides of (\ref{eq:compute-drift}).
Note that $\expect{P_{comp}(t)\left|\right. U(t)}$
can be expressed as follows (using iterated expectations):
\begin{eqnarray*}
&& \hspace{-.5in} \expect{P_{comp}(t)\left|\right. U(t)} \\
&=& \expect{\expect{P_{comp}(t)\left|\right. U(t), A(t), k(t)}\left|\right.U(t)} \\
&=& \expect{ \phi(A(t), k(t)) \left|\right.U(t)}
\end{eqnarray*}
Using this equality in the right hand side of (\ref{eq:final-drift-metric}) and re-arranging terms
yields:
\begin{eqnarray}
&& \hspace{-.2in} \Delta(U(t)) + V\expect{P_{comp}(t) + P_{tran}(t)\left|\right.U(t)} \leq  \nonumber \\
&& B - \expect{U(t) C(P_{tran}(t), S(t)) - VP_{tran}(t) \left|\right. U(t)} \nonumber \\
&& + \expect{U(t)m(A(t), k(t)) + V\phi(A(t), k(t)) \left|\right. U(t)} \label{eq:final-drift-metric2}
\end{eqnarray}

Now note that we have not yet used the properties of the dynamic compression and transmission
policy.  Indeed,
the above expression (\ref{eq:final-drift-metric2}) is a bound that holds for any
compression and transmission scheduling  decisions $k(t) \in \script{K}$, $P_{tran}(t) \in \script{P}$
that are made on slot $t$, including randomized
decisions.  However, note that the dynamic compression and transmission strategy is
designed specifically to minimize the right hand side of (\ref{eq:final-drift-metric2}) over all
alternative decisions that can be made on slot $t$.  Indeed, the compression algorithm
observes $A(t)$ and $U(t)$ and chooses $k(t)\in\script{K}$
to minimize $U(t)m(A(t), k(t)) + V\phi(A(t), k(t))$, which thus minimizes the following  term over all
alternative decisions that can be made on slot $t$:
\[  \expect{U(t)m(A(t),k(t)) + V\phi(A(t), k(t))\left|\right. U(t)} \]
Similarly, the transmission power allocation algorithm is designed to minimize the following
term over all alternative decisions that can be made on slot $t$:
\[  - \expect{U(t) C(P_{tran}(t), S(t)) - VP_{tran}(t) \left|\right. U(t)}  \]
It follows that the right hand side of (\ref{eq:final-drift-metric2}) is less than or equal to the
corresponding expression with $P_{tran}(t)$ and $k(t)$ replaced by $P_{tran}^*(t)$ and
$k^*(t)$, where $P_{tran}^*(t)$ and $k^*(t)$ are any other (possibly randomized)
policies that satisfy $P_{tran}^*(t) \in \script{P}$ and $k^*(t) \in \script{K}$:
\begin{eqnarray}
&\hspace{-.4in} \Delta(U(t)) + V\expect{P_{comp}(t) + P_{tran}(t)\left|\right.U(t)} \leq  \nonumber \\
&B - \expect{U(t) C(P_{tran}^*(t), S(t)) - VP_{tran}^*(t) \left|\right. U(t)} \nonumber \\
&+ \expect{U(t)m(A(t), k^*(t)) + V\phi(A(t), k^*(t)) \left|\right. U(t)} \label{eq:final-drift-metric3}
\end{eqnarray}
Now let $r_1$ be any particular value that satisfies $r_{min} \leq r_1 \leq b\expect{A(t)}$,
and let  $k^*(t)$ be the stationary randomized
policy that yields:
\begin{eqnarray}
\expect{\phi(A(t), k^*(t))} &=& h^*(r_1) \label{eq:aa}\\
\expect{m(A(t), k^*(t))} &=& r_1 \label{eq:bb}
\end{eqnarray}
Such a policy exists by (\ref{eq:stat-h}) and (\ref{eq:stat-r})
of Lemma \ref{lem:exist-particular}.  Similarly, let $r_2$ be any value that
satisfies $0 \leq r_2 \leq r_{max}$, and let $P_{tran}^*(t)$
be the stationary randomized power allocation policy that yields:
\begin{eqnarray}
\expect{C(P_{tran}^*(t), S(t))} &=& r_2 \label{eq:cc} \\
\expect{P_{tran}^*(t)} &=& g^*(r_2) \label{eq:dd}
\end{eqnarray}
Such a policy exists by
(\ref{eq:stat-tran-r}) and (\ref{eq:stat-tran-g}).  Further, the stationary randomized
policies of (\ref{eq:aa})-(\ref{eq:dd}) base decisions only on the current $A(t)$ and
$S(t)$ states, which are i.i.d. over slots (and are hence independent of the current
queue backlog $U(t)$).  Thus, the expectations  of (\ref{eq:aa})-(\ref{eq:dd})
are the same when conditioned on $U(t)$.
Plugging (\ref{eq:aa})-(\ref{eq:dd}) into the right hand side of (\ref{eq:final-drift-metric3}) thus
yields:
\begin{eqnarray}
\Delta(U(t)) + V\expect{P_{comp}(t) + P_{tran}(t)\left|\right.U(t)} \leq \nonumber \\
 B - U(t)(r_2 - r_1) + V(h^*(r_1) + g^*(r_2)) \label{eq:yabba}
\end{eqnarray}

The above inequality holds for all $r_1$ and $r_2$ that satisfy
$r_{min} \leq r_1 \leq b\expect{A(t)}$ and $0 \leq r_2 \leq r_{max}$.
Let $r_1 = r_2 = r^*$, where $r^*$ is the value of $r$ that optimizes the
problem in (\ref{eq:min-h-and-g}) and (\ref{eq:bounds-h-and-g}) of Theorem \ref{thm:min-power},
so that $P_{av}^* = h^*(r^*) + g^*(r^*)$.  Plugging into (\ref{eq:yabba}), we have:
\[ \Delta(U(t)) + V\expect{P_{comp}(t) + P_{tran}(t)\left|\right.U(t)} \leq B  + VP_{av}^* \]
Using the above drift inequality in the Lyapunov Drift Lemma (Lemma \ref{lem:lyap-drift})
and defining $\alpha(t) = VP_{comp}(t) + VP_{tran}(t)$ and $\beta(t) = B + VP_{av}^*$
yields $\overline{P}_{tot} \leq P_{av}^* + B/V$, proving equation (\ref{eq:power-bound})
of Theorem \ref{thm:performance}.

Now choose $r_1 = r_{min}$ and $r_2 = r_{max}$.  Plugging into (\ref{eq:yabba})
and noting that $P_{comp}(t) \geq 0$ and $P_{tran}(t) \geq 0$
gives:
\begin{eqnarray*}
  \Delta(U(t))  \leq
  B  - U(t)(r_{max} - r_{min}) \\
  +  V(h^*(r_{min}) + g^*(r_{max}))
  \end{eqnarray*}
Using the above drift inequality in the Lyapunov Drift Lemma (Lemma \ref{lem:lyap-drift})
and defining $\alpha(t) = U(t)(r_{max} - r_{min})$ and $\beta(t) = B + V(h^*(r_{min}) + g^*(r_{max}))$,
we have:
\[ \limsup_{t\rightarrow\infty} \frac{1}{t} \sum_{\tau=0}^{t-1} \expect{U(\tau)} \leq  \frac{B + V(h^*(r_{min})
+ g^*(r_{max}))}{(r_{max} - r_{min})} \]
The result of (\ref{eq:u-bound}) follows because
$g^*(r_{max}) \leq P_{max}$ and $h^*(r_{min}) \leq \phi_{max}$.  This completes
the proof of Theorem \ref{thm:performance}.

\section{A Simple Delay  Improvement}  \label{section:improvement} 

Here we present a simple improvement to the transmission algorithm 
that can decrease queue backlog while maintaining the exact same 
average power performance specified in Theorem \ref{thm:performance}. 
First observe that the performance theorem (Theorem \ref{thm:performance}) and 
the Lyapunov Drift Lemma (Lemma \ref{lem:lyap-drift}) both specify time average
behavior that is independent of the initial queue backlog.  Indeed, the affects
of the initial condition are transient and decay over time. 
Now suppose the transmission algorithm has the following property: 

\emph{Property 1:} There exists a finite constant $U_{thresh}\geq0$ such that
if $U(0) \geq U_{thresh}$, then $U(t) \geq U_{thresh}$ for all time $t \geq 0$. 

Thus, Property 1 says that if the queue has an initial condition of at least $U_{thresh}$ 
bits, then it will never fall below this threshold of bits.   Clearly Property 1 always
holds with $U_{thresh}=0$.  In any system where Property 1 holds for some constant
$U_{thresh}>0$, then the first $U_{thresh}$ bits in the queue are just acting as a
\emph{place holder} to make $U(t)$ large enough to properly affect the stochastic
optimization.  

\subsection{Example Showing that $U_{thresh}>0$ is Typical} Suppose there is a finite constant $\beta_{max}$ such that: 
\[ C(P, S) \leq \beta_{max} P \: \: \mbox{ for all $P \in \script{P}$ and all $S\in\script{S}$} \]
For example, if the transmission rate function $C(P,S)$ is differentiable with 
respect to $P$, then $\beta_{max}$ can be defined as the largest derivative
with respect to $P$
over all possible channel states.   Because the algorithm chooses $P_{tran}(t)$ 
every slot as the maximizer of $U(t)C(P, S(t)) - VP$ over all $P \in \script{P}$,
it is clear that $P_{tran}(t) = 0$ whenever $U(t)\beta_{max} < V$.  Indeed, we have
for any channel state $S(t)$: 
\[ U(t) C(P, S(t)) - VP \leq [U(t)\beta_{max} - V]P \]
which is maximized only by $P = 0$ if $U(t)\beta_{max} < V$. 
Therefore,  if $U(t) < V/\beta_{max}$, we have $P_{tran}(t) = 0$ and hence 
$\mu(t) = C(P_{tran}(t), S(t)) = 0$, so that the queue backlog cannot further
decrease.  
Define $\mu_{max}$ as the largest possible
transmission rate during a single slot (equal to the maximum of $C(P_{max}, S)$
over all $S \in \script{S}$).   It follows that Property 1 holds in this example with:\footnote{The
$U_{thresh}$ value in (\ref{eq:u-thresh}) satisfies Property 1, but is not necessarily the largest value
that satisfies this property.  Performance can be improved if a larger value that satisfies
Property 1 can be found,  which is often possible for concave rate-power curves defined over
a continuous interval.} 
\begin{equation} \label{eq:u-thresh} 
U_{thresh} \defequiv \max\left[0, \frac{V}{\beta_{max}} - \mu_{max}\right] 
\end{equation} 
The value $U_{thresh}$ determines the number of \emph{place holder bits} 
required in the system. 

\subsection{Delay Improvement Via Place Holder Bits} 

If Property 1 holds for $U_{thresh}>0$, performance can be improved in the following
way: With $U(t)$ being the actual queue backlog, define the \emph{place-holder backlog} 
$\hat{U}(t)$ as follows: 
\[ \hat{U}(t) \defequiv U(t) + U_{thresh} \]
The value $\hat{U}(t)$ can be viewed as backlog that is equal to the actual 
backlog plus $U_{thresh}$ ``fake bits.'' Now assume that $U(0) = 0$, but implement
the Dynamic Compression and Transmission Algorithm using the place-holder
backlog $\hat{U}(t)$ everywhere, instead of the actual queue backlog.  That is, choose
 $k(t) \in \script{K}$ and $P_{tran}(t) \in \script{P}$ as follows: 
\begin{eqnarray*}
 k(t) &=& \arg\min_{k \in \script{K}}[\hat{U}(t) m(A(t), k) + V\phi(A(t), k)] \\
 P_{tran}(t) &=& \arg \max_{P \in \script{P}}[\hat{U}(t) C(P, S(t)) - VP]
 \end{eqnarray*}
 
 With this implementation, we have $\hat{U}(0) = U_{thresh}$, and so 
 $\hat{U}(t) \geq U_{thresh}$ for all $t$ (by Property 1).  It follows that 
 any transmission decisions never take $\hat{U}(t)$ lower than $U_{thresh}$, 
 which is equivalent to saying that all transmission decisions transmit only
 \emph{actual data} (so that $U(t) \geq 0$), rather than \emph{fake data}. 
 The resulting decisions are the same as those in a system with initial backlog 
 of $U_{thresh}$, which yields the same $O(1/V)$ energy performance guarantee as
 before, and yields the same $O(V)$ time average backlog guarantee for the 
 $\hat{U}(t)$ backlog.  However, the \emph{actual} queue backlog is exactly 
 $U_{thresh}$ bits lower than $\hat{U}(t)$ at every instant of time, and so the 
 time average backlog is also exactly $U_{thresh}$ bits lower. Thus, this simple
 improvement yields less actual queue backlog in the system, without any 
 loss in performance.  
 This improvement does not change the $[O(1/V), O(V)]$ 
 tradeoff relation (it simply multiplies the $O(V)$ congestion bound by a smaller
 coefficient), but can yield practical backlog and delay improvements for implementation
 purposes. 
 
 \section{Distortion Constrained Data Compression} \label{section:distortion} 

In the previous sections, we have assumed that all compression options $k \in \script{K}$
are either lossless, or that the distortion they introduce is acceptable.  In this section, we expand the model to allow each compression option to have
its own distortion properties.  Specifically, let $D(t)$ be a non-negative 
real number that represents
a measure of the distortion introduced at time $t$ due to compression.  We assume that this
is a random function of $A(t)$ (the number of packets compressed) and $k(t)$ (the 
compression option), 
and that this function is stationary and independent over all slots with the same 
$A(t)$ and $k(t)$. Define  $d(a, k)$ as the expected distortion function: 
\[ d(a, k) \defequiv \expect{D(t) \left|\right. A(t) = a, k(t) = k} \]

We assume the maximum second moment of distortion is bounded by some
finite constant $\dmax^2$: 
\[ \dmax^2 \defequiv \max_{a \in \{1, \ldots, N\}, k \in \script{K}} \left[  \expect{D(t)^2 \left|\right. A(t) = a, k(t) = k} \right] \]

Assuming that distortion is additive, the goal in this section is to make joint transmission 
and compression actions to minimize time average power expenditure subject to queue
stability and subject to the constraint that time average distortion is bounded by a constant
$d_{av}$: 
\begin{equation} \label{eq:distortion-constraint} 
\lim_{t\rightarrow\infty} \frac{1}{t}\sum_{\tau=0}^{t-1} \expect{D(\tau)} \leq d_{av} 
\end{equation} 
Let $\script{K}$ now represent a set of \emph{extended} compression options, which 
includes lossy compression options (and may also include the maximum distortion
option of throwing away
all data that arrives on slot $t$).   Lossless compression options may still be
available and yield $D(t) = 0$. 

To ensure the distortion constraint (\ref{eq:distortion-constraint}) is satisfied, 
we introduce  a \emph{distortion queue} $X(t)$ that accumulates the total amount of 
distortion in excess of $d_{av}$.  This is similar to the \emph{virtual power queue} 
introduced in \cite{neely-energy-it} for ensuring average power constraints, 
and is an example of a \emph{virtual cost queue} from \cite{now}. 
Specifically, the $X(t)$ queue is implemented purely in software.  It is 
initialized to $X(0)=0$, and is changed from slot to slot according to the following
dynamics: 
\begin{equation} \label{eq:distortion-queue} 
X(t+1) = \max[X(t) - d_{av}, 0] + D(t)
\end{equation} 
where $D(t)$ is the random amount of distortion introduced on slot $t$ (observed
at the end of the compression operation). Stabilizing the $X(t)$ queue ensures the
time average  rate of the $D(t)$ ``arrivals'' is less than or equal to $d_{av}$, which is equivalent
to the distortion constraint (\ref{eq:distortion-constraint}).

\subsection{The Distortion-Constrained Algorithm} 

The queue backlog $U(t)$ evolves as before, according to the dynamics
(\ref{eq:queue-dynamics}).  
The transmission algorithm that selects $P_{tran}(t)\in\script{P}$ every slot 
is the same as before (equation (\ref{eq:transmit-alg1})), 
and the new compression algorithm is given as follows: 

\emph{\underline{Distortion-Constrained Compression Algorithm}:} Every slot
$t$, observe the number of new packet arrivals $A(t)$ and the current queue
backlogs $U(t)$ and $X(t)$, and choose compression option $k(t) \in \script{K}$ 
as follows:
\begin{eqnarray*}
k(t) = \arg \min_{k \in \script{K}} & \left[ U(t)m(A(t), k) + X(t)d(A(t), k) \right. \\
&  \left. + V\phi(A(t), k) \right] 
\end{eqnarray*}
After the compression operation, observe the actual distortion 
$D(t)$ and update $X(t)$ according to (\ref{eq:distortion-queue}). 

We note that the distortion-constrained dynamic compression algorithm
above is useful even in cases when the average power expenditure 
$\phi(A(t), k)$ due to compression is negligible.  This is because 
compression can significantly save transmission power, although intelligent
compression strategies are required to meet the distortion constraints. 

\subsection{Minimum Average Power with Distortion Constraints} 

Define $h_d^*(r)$ as the infimum time average power over all 
compression strategies that make decisions $k(t) \in \script{K}$ as a 
stationary and random function of the observed number of packets $A(t)$, 
subject to a time average output rate of the compressor that is at most $r$ bits/slot, 
and subject to a time average distortion of at most $d_{av}$.  Let $r_{d, min}$ represent
the smallest possible time average output rate of the compressor that yields 
a time average distortion of at most $d_{av}$, optimized over all 
algorithms that choose $k(t) \in \script{K}$ as a stationary and random function of $A(t)$ (and 
hence independently of queue backlog).  We assume that $r_{d, min} < r_{max}$,
so that it is feasible to meet the distortion constraint.  We further assume that it is 
possible to meet the distortion constraint with \emph{strict inequality} while stabilizing
the system.  That is, we assume there is an $\epsilon>0$ such that 
it is possible to achieve a time average distortion rate
of $\overline{D} \leq d_{av} - \epsilon$ using a policy that chooses 
$k(t) \in \script{K}$ as a stationary and random function of $A(t)$, such that 
$\expect{m(A(t), k(t))} \leq r_{max}$. 

Recall that 
$g^*(r)$ is the infimum time average transmission power over all stationary
randomized strategies $P_{tran}(t) \in \script{P}$ that yield time average transmission
rate at least $r$ bits/slot.  Let $P_{av}^*$ represent the minimum time 
average power (summed over compression
and transmission powers) required to stabilize the queueing system subject to the
distortion constraint (\ref{eq:distortion-constraint}). 

\begin{thm} \label{thm:distortion-constrained-min-power}  (Distortion Constrained Minimum
Average Power) The distortion-constrained minimum time average power $P_{av}^*$ is 
equal to the solution of the following optimization problem: 
\begin{eqnarray}
\mbox{Minimize:}  & h_{d}^*(r) +  g^*(r) \label{eq:problem-dist} \\
\mbox{Subject to:} & r_{d, min} \leq r \leq \min[r_{max}, b\expect{A(t)}] \nonumber
\end{eqnarray} 
\end{thm}
\begin{proof}
The proof is similar to the proof of Theorem \ref{thm:min-power} and is omitted for brevity. 
\end{proof} 

\subsection{Lyapunov Analysis} 

Let $\bv{Z}(t) \defequiv [U(t), X(t)]$ be the combined queue state, and define the Lyapunov
function: 
\[ L(\bv{Z}(t)) \defequiv \frac{1}{2}U(t)^2 + \frac{1}{2}X(t)^2 \]
Define the Lyapunov drift: 
\[ \Delta(\bv{Z}(t)) \defequiv \expect{L(\bv{Z}(t+1)) - L(\bv{Z}(t))\left|\right.\bv{Z}(t)} \]
\begin{lem} \label{lem:distortion-drift} For any constant
$V\geq 0$, the Lyapunov drift $\Delta(\bv{Z}(t))$ satisfies
the following for all $t$ and all $\bv{Z}(t)$:
\begin{eqnarray*}
&& \hspace{-.7in} \Delta(\bv{Z}(t)) + V\expect{P_{tot}(t) \left|\right. \bv{Z}(t) } \leq C  \\
&&  - U(t)\expect{C(P_{tran}(t), S(t)) \left|\right.\bv{Z}(t)} \\
&& + U(t) \expect{m(A(t), k(t))\left|\right. \bv{Z}(t)} \\
&&  - X(t)\expect{d_{av} - d(A(t), k(t)) \left|\right.\bv{Z}(t)} \\
&& + V\expect{P_{tran}(t) + \phi(A(t), k(t)) \left|\right.\bv{Z}(t)} \\
\end{eqnarray*}
where the constant $C$ is given by: 
\begin{eqnarray} 
C &\defequiv& \frac{1}{2} \left[d_{av}^2 + \dmax^2  + \sigma^2 + \expect{C(P_{max}, S(t))^2} \right]  \label{eq:C} 
\end{eqnarray} 
\end{lem}
\begin{proof} 
The proof is similar to the proof of 
Lemma \ref{lem:compute-drift} and is omitted for brevity.
\end{proof} 

It can be seen that the Distortion Constrained Compression and Transmission Algorithm
is designed to observe current queue backlogs $X(t)$, $U(t)$ and arrival and channel states
$A(t)$, $S(t)$, and take control actions $k(t) \in \script{K}$, $P_{tran}(t) \in\script{P}$ to 
minimize the right hand side of the drift bound given in Lemma
\ref{lem:distortion-drift}.  

\begin{thm} \label{thm:distortion-performance} (Algorithm Performance with Distortion 
Constraint) Suppose $A(t)$ and $S(t)$ are i.i.d. over slots, and that $r_{d, min} < r_{max}$. 
For any control parameter $V>0$, the Distortion Constrained Compression and Transmission
Algorithm stabilizes the network and satisfies: 
\begin{eqnarray}
\overline{P}_{tot}  &\leq& P_{av}^* + \frac{C}{V} \label{eq:p-dist-bound}  \\
\overline{U}  &\leq& \frac{C + V(P_{max} + \phi_{max})}{(r_{max} - r_{d, min})} \label{eq:u-dist-bound} \\
\overline{D} &\leq& d_{av} \label{eq:d-dist-bound} 
\end{eqnarray}
where $\overline{U}, \overline{P}_{tot}$, and $\overline{D}$ represent $\limsup$ time average
expected queue backlog, total power, and distortion, and the constant $C$ is defined 
in (\ref{eq:C}). 
\end{thm} 
\begin{proof} 
See Appendix D.
\end{proof} 

Thus, the algorithm meets the time average distortion constraint, and the parameter $V$ can be 
used to push total time average power expenditure within $O(1/V)$ of the optimal $P_{av}^*$, 
with an  $O(V)$ tradeoff in average queue congestion $\overline{U}$ and hence average delay.
We note that an improved delay performance can be achieved by 
using $\hat{U}(t) = U(t) + U_{thresh}$ as a  replacement for $U(t)$, 
as described in Section \ref{section:improvement}, 
with $U_{thresh}$ satisfying Property 1, such as the value 
given in (\ref{eq:u-thresh}).

\section{Simulations} \label{section:simulation} 

For simplicity, we consider simulations of the dynamic compression and 
transmission algorithm  of 
Section \ref{section:algorithm} (with the simple improvement of Section 
\ref{section:improvement}), 
without treating
distortion constraints. 
To begin, we first consider a system where the optimal compression 
decision is trivial and does not require a stochastic optimization.   
Specifically, suppose that we have a system where
all three of the following ``Singularity Assumptions''  hold:

\begin{itemize} 
\item The channel is static, so that $S(t)$ is the same for all $t$.
\item The rate-power curve is linear in power, so that $C(P) = \alpha P$ for 
all $P \in \script{P}$, for 
some constant $\alpha$.
\item The raw data arrival rate is less than the maximum transmission rate, that is, 
$b\expect{A(t)} < \alpha P_{max}$. 
\end{itemize} 

In this simple case, the time average transmission power is directly proportional 
to the time average rate of bits transmitted, and so we do not require careful
transmission decisions (all transmissions are equally energy efficient).  
Further,  compression is not required for stability.  It is easy to show in this
special case that the  \emph{exact} 
minimum energy expenditure is achieved by the alternative algorithm of 
observing $A(t)$ every slot $t$
and 
choosing a compression option $\hat{k}(t)\in\script{K}$ as follows:  
\begin{equation} \label{eq:trivial} 
\hat{k}(t) = \arg\min_{k \in \script{K}} \left[\phi(A(t), k) + \frac{1}{\alpha}m(A(t), k)\right]
\end{equation} 
and then transmitting whenever there is a sufficient amount of backlog
to achieve an
efficiency of 
$\alpha$ bits/unit energy. 
That is, we simply choose the compression option that minimizes the sum of the
total energy required to compress and transmit the bits.  A similar observation is used
in \cite{energy-aware-compression} in the study of compression energy ratios  
for popular algorithms.   However, this $\hat{k}(t)$ policy is fragile, in that its optimality 
strongly
relies on all three of the above Singularity Assumptions.   Our dynamic compression and transmission
algorithm is an all-purpose algorithm that should work well for any system, including
systems  that satisfy the above Singularity Assumptions, as well as systems 
that do not.  

\subsection{Scenario I: Singularity Assumptions}  \label{sec:sim:scI}
We first consider a scenario where all three of the 
``Singularity Assumptions'' hold. 
The channel is static with $S(t)=ON$ for all time
slots $t$. The transmit power is constrained to two options
$\script{P}=\{0,1\}$ (we used normalized units of power).  
The rate-power curve is given by $C(P=1) = 2048$ bits/slot, 
and $C(P=0) = 0$. It is clear that the optimal transmission decision 
in this scenario is to transmit only when the queue size is greater 
than or equal to $\mu_{max} = 2048$ bits,  so that all transmissions 
have efficiency $\alpha = \mu_{max}$ bits per unit power.

The wireless link receives data from $8$ different
sensor units. The packet arrival process at each sensor is i.i.d.
over slots and follows a Bernoulli distribution with the probability
of an arrival $p=\frac{1}{2}$. We fix the packet size, $b$, to $256$
bits. Hence, the arrival process for the wireless link, $A(t)$,
follows a Binomial distribution with parameters $(8,\frac{1}{2})$
with an average arrival rate of $b\expect{A(t)} = 1024$ bits/slot.
Note that in this case we have  $b\expect{A(t)} < \mu_{max}$, 
and so compression is not needed  for queue stability. 

Two compression options are available to the link controller
($\script{K}=\{0,1\}$). For $A(t)=a$ and $k(t)=1$, the size of the
data after compression, $R(t)$, is uniformly distributed in
$[\frac{2ab}{5},\frac{3ab}{5}]$ and the compression power 
is uniformly distributed in $[0.45, 0.55]$. Hence, the
average compressed output is $m(a,1)=\frac{ab}{2}$ with an average
power of $\phi(a,1)=0.5$.
In this scenario, compression is energy-expensive compared to transmission, 
and, because compression is not required for stability,  it is easy 
to see from (\ref{eq:trivial})  
that  transmitting all data without compression is optimal.  Thus, the 
policy of (\ref{eq:trivial}) has  $\hat{k}(t) = 0$ for
all $t$, and transmits whenever $U(t) \geq 2048$, yielding an optimal
average power $\overline{P}_{tot} = b\expect{A(t)}/\mu_{max} = 0.5$.

\begin{figure} [t]
\centering
  \includegraphics[height=2in, width=3in]{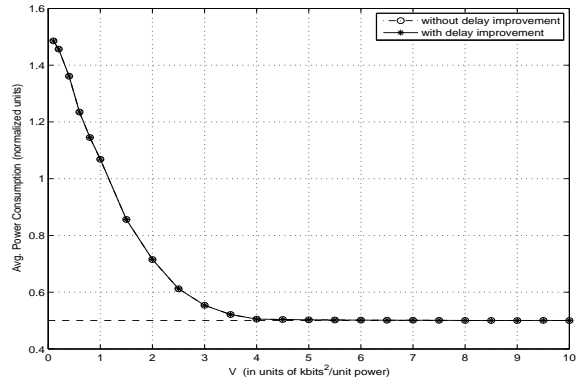} 
  \caption{Avg. Power expenditure vs. V (two nearly identical curves are shown).}
  \label{fig:Pavg}
\end{figure}

\begin{figure}[t]
\centering
  \includegraphics[height=2in, width=3in]{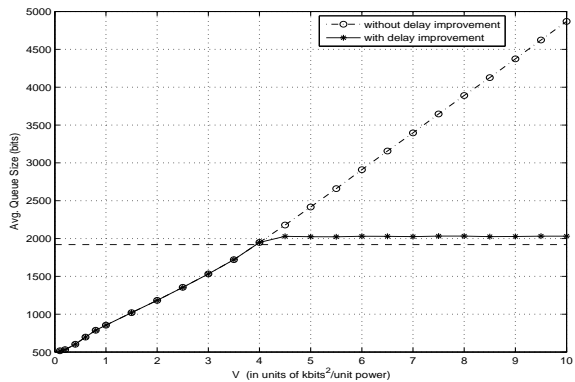} 
  \caption{Avg. Queue size vs. V}
  \label{fig:Qavg}
\end{figure}

 We simulate our dynamic compression and transmission algorithm over
 $10^6$ slots, for various choices of the 
 $V$ parameter.  
 Fig. \ref{fig:Pavg} shows that average power 
indeed converges to 
$0.5$ as $V$ is increased.  Fig. \ref{fig:Pavg} also shows that, as expected,
incorporating the simple delay improvement of 
Section \ref{section:improvement} does not affect power expenditure.  
Fig. \ref{fig:Qavg} shows the time average queue backlog versus
$V$. For simulations without the delay improvement, the queue backlog increases linearly with $V$.
The delay improvement uses $U_{thresh} = \max[V/\mu_{max} - \mu_{max},0]$, and reduces
queue backlog (maintaining  a relatively constant average backlog 
for $V \geq 5000$).   The dashed horizontal line 
at $\overline{U} = 1920$ bits (shown in Fig. \ref{fig:Qavg})
is the average queue size obtained by the $\hat{k}(t)$ policy that performs no compression
and transmits only when $U(t) \geq \mu_{max}$.

\subsection{Scenario II: Compression for Stability} 
We next consider the same scenario, but increase the raw data rate beyond
$\mu_{max}$, so that  compression is required 
for queue stability (this removes the third ``Singularity Assumption'').
However, the proper fraction of time to compress may be different for each 
observed $A(t)$ value, and in general it depends on the distribution of the arrival
process $A(t)$.  Our dynamic algorithm optimizes without this statistical knowledge,
learning the correct  actions for each observed $A(t)$ value.

Fig. \ref{fig:Pavg_arr} shows the increase in average power
expenditure for our algorithm (with delay improvement) 
as the raw arrival rate increases. This raw arrival rate
is increased by adjusting the packet size $b$ from $256$ to $1024$
(the parameter
$V$ is fixed to $V = 10$ $\mbox{kbits}^2/$unit power,  
so that $U_{thresh} = \max[V/\mu_{max}  - \mu_{max}, 0] \approx 2835$ bits, 
and the simulation time for each data point
is five million slots).   Also shown is the average power when there is no compression
but when the same dynamic transmission strategy is used.
For arrival rates $b\expect{A(t)} > 1024$, compression is required for energy efficiency, 
and for $b\expect{A(t)} >  \mu_{max} = 
2048$,  compression is  required for both energy efficiency and stability.   
For $V=10$, our dynamic algorithm yields energy efficiency within roughly 
$0.4\%$ of optimal for the rate region tested.  
For example, when the raw arrival rate is 3400, the
optimum is $P_{av}^* = 1.310$ (achievable by compressing whenever 
$A(t) \geq 3$), and our algorithm achieves $\overline{P}_{tot} = 1.314$. 
Because compression reduces data on average by a factor of 2, the maximum raw
arrival rate that can be stably supported is
$2\mu_{max}$. When $b\expect{A(t)} \geq 2\mu_{max}$, our algorithm
learns to compress \emph{all} data,  leading to
an average power expenditure of $1.5$ ($0.5$ power units for compression, plus
$1$ unit for transmission).

\begin{figure}[t]
\centering
  \includegraphics[height=2in, width=3in]{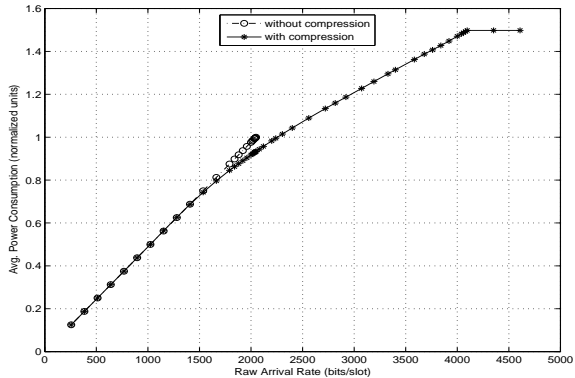} 
  \caption{Avg. Power Consumption vs. Raw Arrival Rate.  Data points for the 
  experiments without compression are shown only for the stable region, i.e., 
  for raw arrival rates less than 2048.}
  \label{fig:Pavg_arr}
\end{figure}

\begin{figure}[t]
\centering
  \includegraphics[height=2in,width=3in]{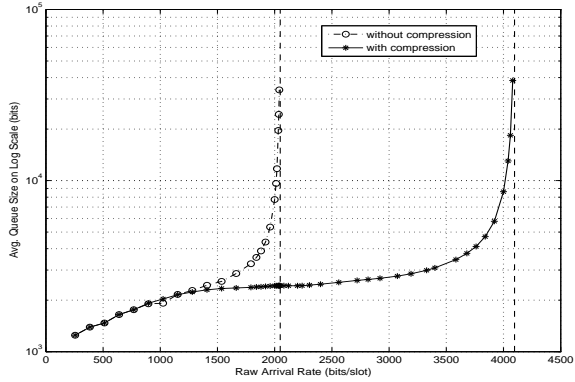} 
  \caption{Avg. Queue size (bits) vs. Raw Arrival Rate.}
  \label{fig:Qavg_arr}
\end{figure}

Fig. \ref{fig:Qavg_arr} shows the change in average queue size as
the raw data arrival rate is increased.  Without any data compression, 
the average queue backlog grows to infinity as the raw data rate 
approaches the vertical asymptote $2048$ bits/slot. With the
data compression option the queue size remains quite flat beyond this threshold, 
increasing at a new vertical asymptote at $4096$ bits/slot.

\subsection{Scenario III:  Nonlinear Rate-Power Curve}
\label{sec:sim:scII} For this scenario, the Bernoulli arrival process is the
same as in Scenario I (Section \ref{sec:sim:scI}), with packet size $b = 256$ bits. 
However, we make
the following changes:
\begin{itemize}
\item The rate-power curve is non-linear in power with $C(P) = \alpha
\log(1+\beta P)$ for transmit power $P$, where $P$ is any real number
in the interval $0 \leq P \leq P_{max}$.\footnote{The $\log()$ used here denotes a natural
logarithm.} 
\item The raw data arrival rate is less than the maximum
transmission rate, i.e. $b\expect{A(t)}  = 1024 < \mu_{max} \defequiv \alpha \log(1+\beta P_{max})$.
\item The compressed data 
was taken from a trace of experimental 
data from \cite{VTB}, and was compressed using
the zlib compression library \cite{www-zlib}.   
\end{itemize}


A single compression option ($\script{K}=\{0,1\}$) is available at
the transmitting node.  For $k=1$, we have $m(a,k) = \frac{ab}{1.1}$ for $A(t) = a \leq 3$ packets, 
and $m(a,k) = \frac{ab}{1.5}$ for $A(t) = a > 3$.  These average compression ratios were 
obtained from 
the experimental data in  \cite{VTB} using
the zlib data compression library \cite{www-zlib}. The work in 
\cite{VTB} considers a wireless sensor network where each node senses
vibrations of a large suspension bridge.
We assume the 
power expenditure during compression, $\phi(a,1)$, is $5$ units for
$A(t)=a \leq 3$  and $8$ units for $A(t)=a > 3$.
For transmission, we 
use $P_{max} = 750$ power units, $\alpha = 1060$, and $\beta=1/16$, so
that $\mu_{max} = C(P_{max}) \approx 4100$ bits/slot.

\begin{figure}[t]
\centering
  \includegraphics[height=2in, width=3in]{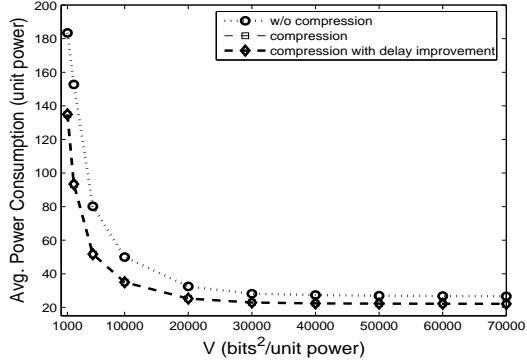} 
  \caption{Nonlinear Rate-Power Curve: Avg. energy usage vs. V.  The two curves with 
  compression (with and without delay improvement) are almost identical.}
  \label{fig:trace_Pavg}
\end{figure}

\begin{figure}[t]
\centering
  \includegraphics[height=2in, width=3in]{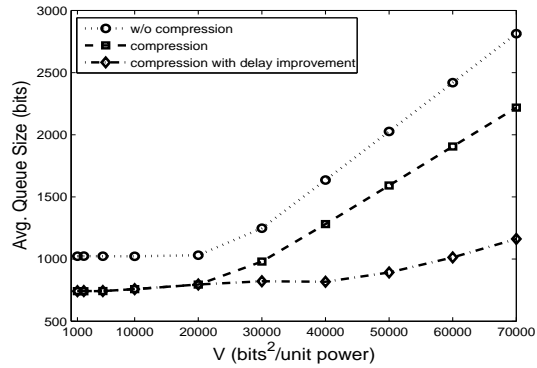} 
  \caption{Nonlinear Rate-Power Curve: Avg. Queue size vs. V}
  \label{fig:trace_Qavg}
\end{figure}

Fig. \ref{fig:trace_Pavg} shows that the average power consumption
for our dynamic algorithm converges to $22.21$.  As expected, the power 
curves are almost identical
with and without delay improvement.
Fig. \ref{fig:trace_Pavg} also shows the
average power expenditure converges to $26.042$ if all data is transmitted
uncompressed (but still using our transmission strategy of (\ref{eq:transmit-alg1})).
Our dynamic compression algorithm yields an
energy savings between $15$ and $25$ percent 
across the $V$ range tested, as  
compared to sending all the data
uncompressed.
Fig. \ref{fig:trace_Qavg} shows how the average
queue backlog increases with $V$.  The plot without delay improvement 
uses $U_{thresh} = 0$ (which is also the value of $U_{thresh}$ that would
be given by (\ref{eq:u-thresh}) for the $V$ range tested),\footnote{The value of $U_{thresh}$ given by 
(\ref{eq:u-thresh}) uses 
$\beta_{max} = \alpha\beta$ and $\mu_{max} = C(P_{max})$.} and the plot with delay improvement
uses $U_{thresh}$ as the largest 
value that satisfies  Property 1 under the given transmission policy (different from
the value given by (\ref{eq:u-thresh})). This value of $U_{thresh}$ for 
logarithmic rate-power curves is derived in Appendix E (see equation 
(\ref{eq:uthresh-log})). Note that for
 $\alpha = 1060$, $P_{max} = 750$, $\beta = 1/16$, $V  = 70000$, we have 
$U_{thresh} = 1056.6$ bits, a quite significant improvement in queue backlog with 
no loss of power efficiency.

\section{Conclusion} 

This paper presents a dynamic decision technique for joint compression and 
transmission in a wireless node, using results of stochastic network 
optimization.  The approach allows
total average power expenditure to be 
pushed arbitrarily close to optimal,  with a corresponding
delay (and queue congestion) tradeoff.  The resulting compression
and transmission algorithms are simple to implement and  operate well in a variety
of settings.   We believe this approach  will also 
be useful for management of compression and sensing in 
multi-hop networks, where energy, stability, and delay issues will become
increasingly important in future applications.

\section*{Appendix A -- Proof of Lemma \ref{lem:exist-particular}}
\begin{proof} (Lemma \ref{lem:exist-particular})
The function $h^*(r)$ is defined in terms of an infimum of $\expect{P_{comp}(t)}$ over all
stationary randomized policies that yield $\expect{R(t)} \leq r$. It follows that there exists
an infinite sequence of stationary randomized policies, indexed by integers
$i\in \{1, 2, \ldots\}$, having expectations
$\expect{R^{(i)}(t)}$ and $\expect{P_{comp}^{(i)}(t)}$ that satisfy:
\begin{eqnarray}
&\expect{R^{(i)}(t)} \leq r  \mbox{ for all $i\in\{1, 2, \ldots \}$}& \label{eq:fum1}\\
&\lim_{i\rightarrow\infty} \expect{P_{comp}^{(i)}(t)} =  h^*(r)& \label{eq:fum2}
\end{eqnarray}
However, each policy $i$ is defined in terms of a collection  of
probabilities $(\gamma_{a,k}^{(i)})$ for
$a \in \{0, 1, \ldots, N\}$ and $k \in \script{K}$.  This collection of probabilities
can be viewed as a finite dimensional vector that is
contained in a compact set $\Omega$
defined by the constraints (\ref{eq:gamma1})-(\ref{eq:gamma2}).
The compact set $\Omega$
contains its limit points, and hence the infinite sequence $\{(\gamma_{a,k}^{(i)})\}_{i=1}^{\infty}$
contains a convergent subsequence that converges to a point
$(\gamma_{a,k}^*)  \in \Omega$. This point is a vector of probabilities that  define a
stationary randomized algorithm with expectations
$\expect{R^*(t)}$ and $\expect{P_{comp}^*(t)}$.
Now recall that a general stationary randomized
algorithm  yields  expectations $\expect{R(t)}$ and $\expect{P_{comp}(t)}$
that can be expressed as linear (and hence continuous) function of the
probabilities $(\gamma_{a,k})$, as shown in the left hand sides of equations
(\ref{eq:statphi}) and (\ref{eq:statm}). Hence, the properties
(\ref{eq:fum1}) and (\ref{eq:fum2}) are preserved in the limit, so that
$\expect{R^*(t)} \leq r$ and  $\expect{P_{comp}^*(t)} = h^*(r)$.

If $\expect{R^*(t)} = r$, then we are done. Else, we have
$\expect{R^*(t)}  < r \leq b\expect{A(t)}$ (recall that $r \leq b\expect{A(t)}$ by
assumption in the statement of  Lemma \ref{lem:exist-particular}). Hence, $r = \theta\expect{R^*(t)} + (1-\theta)b\expect{A(t)}$ for some probability $\theta$. Note that the $0$-power algorithm
of no compression yields an expected compression output of exactly $b\expect{A(t)}$.  
It follows that defining $R'(t)$ as the stationary randomized policy that chooses 
$R^*(t)$ with probability $\theta$ and chooses no compression 
with probability $(1-\theta)$ yields $\expect{R'(t)} = r$. This policy $R'(t)$ cannot
use more power than policy $R^*(t)$, and hence $\expect{P_{comp}'(t)} \leq h^*(r)$. 
But we also have $h^*(r) \leq \expect{P_{comp}'(t)}$, because  $h^*(r)$ is defined as the 
infimum average power over all stationary randomized 
policies that yield a compressor output rate of at most 
$r$.   
\end{proof}

\section*{Appendix B  -- Proof of Theorem \ref{thm:min-power}}

Here we prove Theorem \ref{thm:min-power}.  Consider any policy that
stabilizes the queue, and let $k(t)$ and $P_{tran}(t)$ be the resulting compression
and transmission power decisions chosen over time (where $k(t) \in \script{K}$ and
$P_{tran}(t) \in \script{P}$ for all $t$).  Let  $R(t) = \Psi(A(t), k(t))$
be the resulting bit output process from the compressor, and let  $P_{comp}(t)$ be the
resulting compression power expenditure process.  Let $\mu(t) = C(P_{tran}(t), S(t))$ be
the transmission rate process.   We want to show that:
\begin{equation} \label{eq:want-to-show}
\limsup_{t\rightarrow\infty} \frac{1}{t} \sum_{\tau=0}^{t-1} \expect{P_{comp}(\tau) + P_{tran}(\tau)} \geq P_{av}^*
\end{equation}
where $P_{av}^*$ is defined in Theorem \ref{thm:min-power}.
We have two preliminary lemmas.

\begin{lem} \label{lem:appb-1} Suppose there are constants $r$ and $\overline{P}_c$ together
with an infinite sequence
of times  $\{t_i\}_{i=1}^{\infty}$ such that:
\begin{eqnarray}
 \lim_{t_i\rightarrow\infty} \frac{1}{t_i} \sum_{\tau=0}^{t_i-1} \expect{R(\tau)}  &=& r \label{eq:lemmab1}\\
 \lim_{t_i \rightarrow\infty} \frac{1}{t_i} \sum_{\tau=0}^{t_i-1} \expect{P_{comp}(\tau)} &=& \overline{P}_c \label{eq:lemmab2}
\end{eqnarray}
Then $\overline{P}_c \geq h^*(r)$.
\end{lem}
\begin{proof}
The proof is given in Appendix C.
\end{proof}

\begin{lem} \label{lem:appb-2} Suppose there are constants $\overline{\mu}$ and $\overline{P}_t$ together
with an infinite sequence of times $\{t_i\}_{i=1}^{\infty}$ such that:
\begin{eqnarray}
\lim_{t_i\rightarrow\infty} \frac{1}{t_i} \sum_{\tau=0}^{t_i-1} \expect{P_{tran}(\tau)} &=& \overline{P}_t \label{eq:appb-3}\\
\lim_{t_i\rightarrow\infty} \frac{1}{t_i} \sum_{\tau=0}^{t_i-1} \expect{\mu(\tau)} &=& \overline{\mu} \label{eq:appb-4}
\end{eqnarray}
Then $\overline{P}_t \geq g^*(\overline{\mu})$.
\end{lem}
 \begin{proof}
The proof is given in Appendix C. 
 \end{proof}

 Now define $\overline{P}_{tot}$ as the $\limsup$ total power expenditure given by the
 left hand side of inequality (\ref{eq:want-to-show}).  Let $\tilde{t}_i$ be an infinite
  subsequence of times over which the $\limsup$ is achieved, so that:
  \begin{equation} \label{eq:total-sup}
  \lim_{\tilde{t}_i\rightarrow\infty} \frac{1}{\tilde{t}_i} \sum_{\tau=0}^{\tilde{t}_i-1} \expect{P_{comp}(\tau)
  + P_{tran}(\tau)}  = \overline{P}_{tot}
  \end{equation}
   Now define:
  \begin{eqnarray*}
  \overline{R}(t) = \frac{1}{t} \sum_{\tau=0}^{t-1} \expect{R(\tau)}  \:  \: , \: \:
  \overline{P}_{comp}(t) = \frac{1}{t} \sum_{\tau=0}^{t-1} \expect{P_{comp}(\tau)} \\
  \overline{P}_{tran}(t) = \frac{1}{t} \sum_{\tau=0}^{t-1} \expect{P_{tran}(\tau)} \: \: , \: \:
  \overline{\mu}(t) = \frac{1}{t} \sum_{\tau=0}^{t-1} \expect{\mu(\tau)}
  \end{eqnarray*}
  and note that for all timeslots $t$ we have:
  \begin{eqnarray*}
  0 \leq \overline{R}(t) \leq b\expect{A(t)}  \: \: , \: \: 0 \leq \overline{P}_{comp}(t) \leq \phi_{max} \\
  0 \leq \overline{P}_{tran}(t) \leq P_{max} \: \: ,  \: \: 0 \leq \overline{\mu}(t) \leq r_{max}
  \end{eqnarray*}
 It follows that  $(\overline{R}(\tilde{t}_i), \overline{P}_{comp}(\tilde{t}_i),
 \overline{P}_{tran}(\tilde{t}_i), \overline{\mu}(\tilde{t}_i))$ can be viewed as an infinite sequence
 contained in a four dimensional compact set, and thus has a convergent subsequence.
 Let $\{t_i\}$ represent the convergent subsequence of times, so that there exist
 constants $r$, $\overline{P}_c$, $\overline{P}_t$, and $\overline{\mu}$ such that:
  \begin{eqnarray*}
  \lim_{t_i\rightarrow\infty}  \frac{1}{t_i} \sum_{\tau=0}^{t_i-1} \expect{R(\tau)}  =  r \\
  \lim_{t_i\rightarrow\infty} \frac{1}{t_i} \sum_{\tau=0}^{t_i-1} \expect{P_{comp}(\tau)}  = \overline{P}_c \\
  \lim_{t_i\rightarrow\infty}  \frac{1}{t_i} \sum_{\tau=0}^{t_i-1} \expect{P_{tran}(\tau)} = \overline{P}_t \\
 \lim_{t_i\rightarrow\infty}  \frac{1}{t_i} \sum_{\tau=0}^{t_i-1} \expect{\mu(\tau)}  = \overline{\mu}
  \end{eqnarray*}
  Furthermore, because $\{t_i\}_{i=1}^{\infty}$ is an infinite
  subsequence of the original sequence $\{\tilde{t}_i\}$, we have from
  (\ref{eq:total-sup}) that $\overline{P}_c + \overline{P}_t = \overline{P}_{tot}$.
  From Lemmas \ref{lem:appb-1} and \ref{lem:appb-2} we must have the following:
  \begin{eqnarray*}
  \overline{P}_c &\geq& h^*(r) \\
  \overline{P}_t &\geq& g^*(\overline{\mu})
  \end{eqnarray*}
  Therefore:
  \begin{equation}
  \overline{P}_{tot} \geq h^*(r) + g^*(\overline{\mu})  \label{eq:viola1}
  \end{equation}

  We now use the fact that queue $U(t)$ is stable.  It is known that a stable queue must
  satisfy (see \cite{neely-energy-it} \cite{now}):
  \begin{equation} \label{eq:zz}
   \lim_{t\rightarrow\infty} \frac{\expect{U(t)}}{t} = 0
   \end{equation}
  However, note that for all times $t_i$ we have:
  \[ U(t_i) \geq \sum_{\tau=0}^{t_i - 1} R(\tau)- \sum_{\tau=0}^{t_i-1} \mu(\tau) \]
  This is true because the total unfinished work in the system at time $t_i$ is
  no more than the total bit arrivals minus the maximum possible bit departures
  during the interval from $0$ to $t_i -1$.  Therefore (taking an expectation and dividing by $t_i$):
  \[ \frac{\expect{U(t_i)}}{t_i} \geq \frac{1}{t_i} \sum_{\tau=0}^{t_i-1} \expect{R(\tau)} - \frac{1}{t_i} \sum_{\tau=0}^{t_i-1} \expect{\mu(\tau)} \]
  Taking a limit of the above expression as
  $t_i \rightarrow \infty$, and using (\ref{eq:zz}) yields $0 \geq r - \overline{\mu}$.  Therefore,
  queue stability implies that $r \leq \overline{\mu}$.  Because the function $g^*(r)$ is non-decreasing, it
  follows that  $g^*(r) \leq g^*(\overline{\mu})$.  Using this fact in (\ref{eq:viola1}) yields:
  \[ \overline{P}_{tot} \geq h^*(r) + g^*(r) \]
  Furthermore, it is not difficult to show that the values of $r$ and $\overline{\mu}$
  must satisfy $r_{min} \leq r \leq b\expect{A(t)}$ and
  $0 \leq \overline{\mu} \leq r_{max}$.
  Because $r \leq \overline{\mu}$, it follows that:
  \begin{equation} \label{eq:almost-done}
  r_{min} \leq r \leq \min[r_{max}, b\expect{A(t)}]
  \end{equation}
  Therefore, the value of $h^*(r) + g^*(r)$ is greater than or equal to the
  \emph{minimum} value of this quantity, minimized  over all $r$ that satisfy
  the constraint (\ref{eq:almost-done}), which is the definition of  $P_{av}^*$.
 Therefore:
 \[ \overline{P}_{tot} \geq h^*(r) + g^*(r) \geq P_{av}^* \]
 This proves Theorem \ref{thm:min-power}.

  \section*{Appendix C -- Proof of Lemmas \ref{lem:appb-1} and \ref{lem:appb-2}}

  \begin{proof} (Lemma \ref{lem:appb-1})
  Here we prove Lemma \ref{lem:appb-1}.  Suppose that (\ref{eq:lemmab1}) and (\ref{eq:lemmab2})
  hold. For all timeslots $t$, we have (by iterated expectations):
  \begin{eqnarray*}
  \expect{R(t)} &=& \expect{\expect{R(t)\left|\right.A(t), k(t)}}\\
  &=& \expect{m(A(t), k(t))}
  \end{eqnarray*}
  Similarly, we have $\expect{P_{comp}(t)} = \expect{\phi(A(t), k(t))}$ for all $t$.
  The equations (\ref{eq:lemmab1}) and (\ref{eq:lemmab2}) thus become:
  \begin{eqnarray}
  \lim_{t_i \rightarrow\infty} \frac{1}{t_i} \sum_{\tau=0}^{t_i-1} \expect{m(A(\tau), k(\tau))} &=& r \label{eq:foo1}\\
  \lim_{t_i \rightarrow\infty} \frac{1}{t_i} \sum_{\tau=0}^{t_i-1} \expect{\phi(A(\tau), k(\tau))} &=& \overline{P}_c \label{eq:foo2}
  \end{eqnarray}
  For any time $t$ we have:
  \begin{eqnarray*}
  && \hspace{-.3in} \frac{1}{t}\sum_{\tau=0}^{t-1} \expect{m(A(\tau), k(\tau))} \\
  &=& \sum_{a=0}^N\sum_{k\in\script{K}} \frac{1}{t} \sum_{\tau=0}^{t-1} m(a,k)p_A(a)Pr[k(\tau)=k\left|\right. A(\tau)=a]\\
  &=& \sum_{a=0}^N\sum_{k\in\script{K}} m(a,k) p_A(a)\gamma_{a,k}(t)
  \end{eqnarray*}
  where we define probabilities $(\gamma_{a,k}(t))$ as follows:
  \[ \gamma_{a,k}(t) \defequiv \frac{1}{t} \sum_{\tau=0}^{t-1} Pr[k(\tau)=k\left|\right.A(\tau)=a] \]
  Similarly, for any time $t$ we have:
  \begin{eqnarray*}
\frac{1}{t}\sum_{\tau=0}^{t-1} \expect{\phi(A(\tau), k(\tau))} =  \sum_{a=0}^N\sum_{k \in \script{K}} \phi(a,k)p_A(a)\gamma_{a,k}(t)
  \end{eqnarray*}
  It follows from (\ref{eq:foo1}) and (\ref{eq:foo2}) that:
  \begin{eqnarray}
   \lim_{t_i \rightarrow\infty} \sum_{a=0}^{N}\sum_{k\in\script{K}} p_A(a)m(a,k) \gamma_{a,k}(t_i) &=& r \label{eq:foo3} \\
    \lim_{t_i \rightarrow\infty} \sum_{a=0}^{N}\sum_{k\in\script{K}} p_A(a)\phi(a,k) \gamma_{a,k}(t_i) &=& \overline{P}_c \label{eq:foo4}
  \end{eqnarray}
 It is straightforward to show that the probabilities $(\gamma_{a,k}(t))$ satisfy the following
 constraints for all $t$:
 \begin{eqnarray}
 \gamma_{a,k}(t) \geq 0 \: \: \mbox{ for all $a, k$} \label{eq:compact1}\\
 \sum_{k\in\script{K}} \gamma_{a,k}(t) = 1 \: \: \mbox{ for all $a$}  \label{eq:compact2}
 \end{eqnarray}
 The above constraints imply that $(\gamma_{a,k}(t))$ can be viewed as a vector
 of values contained in a finite dimensional compact set for all $t$.
 It follows that $\{(\gamma_{a,k}(t_i))\}$ forms an infinite sequence of probability
 vectors contained in a compact set, and so there must exist a convergent subsequence
 of times $\{t_i'\}$ for which $\{(\gamma_{a,k}(t_i'))\}$ converges to a point
 $(\gamma_{a,k}^*)$ contained in the set.  Therefore:
 \begin{eqnarray}
 \gamma_{a,k}^* \geq 0 \: \: \mbox{ for all $a,k$}  \label{eq:endy1} \\
  \sum_{k\in\script{K}} \gamma_{a,k}^* = 1 \: \: \mbox{ for all $a$} \label{eq:endy2} \\
  \sum_{a=0}^N\sum_{k\in\script{K}} p_A(a)m(a,k) \gamma_{a,k}^* = r \label{eq:endy3} \\
  \sum_{a=0}^N\sum_{k\in\script{K}} p_A(a)\phi(a,k) \gamma_{a,k}^* = \overline{P}_c \label{eq:endy4}
  \end{eqnarray}
  where (\ref{eq:endy1}) and (\ref{eq:endy2}) follow because $(\gamma_{a,k}^*)$ is a limit
  point of the compact set defined by (\ref{eq:compact1}) and (\ref{eq:compact2}) and hence
  is an element of that set.  Equalities (\ref{eq:endy3}) and (\ref{eq:endy4}) follow
  because $\{t_i'\}$ is an infinite subsequence of the original sequence of times
  $\{t_i\}$, and hence the same limits in (\ref{eq:foo3}) and (\ref{eq:foo4}) are preserved
  when taken over this subsequence.

  Because $(\gamma_{a,k}^*)$ and $\overline{P}_c$ satisfy (\ref{eq:endy1})-(\ref{eq:endy4}),
  these values define a \emph{particular solution} for the constraints (\ref{eq:statphi})-(\ref{eq:gamma2})
  of the optimization problem of Definition \ref{def:1} in Section \ref{section:optimal-definition}.
  Therefore, $\overline{P}_c$ is greater than or equal to the infimum value of power, infimized
  over all solutions that satisfy these constraints, which is defined as $h^*(r)$. That is,
  $\overline{P}_c \geq h^*(r)$.  This completes the proof of Lemma \ref{lem:appb-1}.
  \end{proof}
  
  
  \begin{proof} (Lemma \ref{lem:appb-2})   Here we prove Lemma \ref{lem:appb-2}.
  Suppose that (\ref{eq:appb-3}) and (\ref{eq:appb-4}) hold.  Similar to the proof of
  Lemma \ref{lem:appb-1}, we can show that for any timeslot $t$:
  \begin{eqnarray}
  \frac{1}{t}\sum_{\tau=0}^{t-1}\expect{\mu(\tau)} &=& \sum_{s \in\script{S}} \pi_s \overline{\mu}_s(t) \label{eq:appc-1} \\
  \frac{1}{t}\sum_{\tau=0}^{t-1} \expect{P_{tran}(\tau)} &=& \sum_{s \in \script{S}} \pi_s \overline{P}_s(t)\label{eq:appc-2}
  \end{eqnarray}
  where $\overline{\mu}_s(t)$ and $\overline{P}_s(t)$ are defined for all $s \in \script{S}$ as follows:
  \begin{eqnarray*}
  \overline{\mu}_s(t) &\defequiv& \frac{1}{t} \sum_{\tau=0}^{t-1} \expect{C(P_{tran}(\tau), s) \left|\right. S(t) = s} \\
  \overline{P}_s(t) &\defequiv& \frac{1}{t} \sum_{\tau=0}^{t-1} \expect{P_{tran}(\tau)\left|\right. S(t) = s}
  \end{eqnarray*}
  For each timeslot $t$ and each channel state
  $s \in \script{S}$, the values $(\overline{\mu}_s(t), \overline{P}_s(t))$ defined above
  are in the convex hull of the set $\Omega_s$ defined below:
  \[ \Omega_s \defequiv \{ (\mu, p) \left|\right. p \in \script{P} \:  \mbox{ and } \:  \: \mu = C(p, s)\}  \]
  The set $\Omega_s$ is 2-dimensional.  It follows by Caratheodory's theorem \cite{bertsekas-convex}
  that any element $(\overline{\mu}_s(t), \overline{P}_s(t))$ contained in the convex hull of
  $\Omega_s$ can be expressed as a convex combination of at most $3$ elements of $\Omega_s$.
  Thus, there exist powers $P_{s,z}(t) \in \script{P}$ and probabilities $\alpha_{s,z}(t)$ such that:
  \begin{equation} \label{eq:appc3}
   (\overline{\mu}_s(t), \overline{P}_s(t))  = \sum_{z=1}^3 \alpha_{s,z}(t) (C(P_{s,z}(t), s), P_{s,z}(t))
   \end{equation}
where $\sum_{z=1}^3 \alpha_{s,z}(t)  = 1$ for all $s, t$.

Using (\ref{eq:appc3}) and (\ref{eq:appc-1}),(\ref{eq:appc-2}), the limit equations of (\ref{eq:appb-3})
and (\ref{eq:appb-4}) become:
\begin{eqnarray*}
\lim_{t_i\rightarrow\infty}  \sum_{s\in\script{S}} \sum_{z=1}^3 \pi_s \alpha_{s,z}(t_i) C(P_{s,z}(t_i), s) &=& \overline{\mu} \\
\lim_{t_i \rightarrow\infty} \sum_{s\in\script{S}} \sum_{z=1}^3 \pi_s \alpha_{s.z}(t_i) P_{s,z}(t_i) &=& \overline{P}_t
\end{eqnarray*}

 It follows that for any $\epsilon>0$, there exists a stationary randomized policy for choosing
 $P_{tran}(t)$ as a random function of the observed channel state $S(t)$ such that:
 \begin{eqnarray*}
 \expect{C(P_{tran}(t), S(t))} &\geq& \overline{\mu} - \epsilon \\
 \expect{P_{tran}(t)} &\leq& \overline{P}_t + \epsilon
 \end{eqnarray*}
  This stationary policy can be modified to create another stationary randomized policy
  that has average transmission rate greater than or equal to $\overline{\mu}$ simply by independently
  choosing $P_{tran}(t)  = P_{max}$ every timeslot with some small probability.  Thus,
  for any value $\delta>0$, there exists a stationary randomized policy for
  choosing $P_{tran}(t)$ that yields:
   \begin{eqnarray*}
 \expect{C(P_{tran}(t), S(t))} &\geq& \overline{\mu}  \\
 \expect{P_{tran}(t)} &\leq& \overline{P}_t + \delta
 \end{eqnarray*}
It follows that $\expect{P_{tran}(t)}$ in the above policy satisfies
$\expect{P_{tran}(t)} \geq g^*(\overline{\mu})$, because $g^*(\overline{\mu})$ is
defined as the smallest average power over the class of stationary randomized
algorithms that support an average transmission rate of at least $\overline{\mu}$
(see Definition \ref{def:2} in Section \ref{section:optimal-definition}).
 Therefore:
 \[ g^*(\overline{\mu}) \leq \expect{P_{tran}(t)} \leq \overline{P}_t + \delta \]
 The above inequality holds for all $\delta >0$, and so $g^*(\overline{\mu}) \leq \overline{P}_t$,
 which completes the proof of Lemma \ref{lem:appb-2}.
  \end{proof}

\section*{Appendix D -- Proof of Theorem \ref{thm:distortion-performance}}

Here we prove that the Distortion Constrained Compression and Transmission Algorithm
yields performance as given in Theorem \ref{thm:distortion-performance}. Because
the algorithm observes the current network state and 
makes control decisions $k(t) \in \script{K}$, $P_{tran}(t)\in\script{P}$
that minimize the right hand side of the drift bound given in Lemma \ref{lem:distortion-drift}, 
we have: 
\begin{eqnarray}
&& \hspace{-.7in} \Delta(\bv{Z}(t)) + V\expect{P_{tot}(t) \left|\right. \bv{Z}(t) } \leq C \nonumber \\
&&  - U(t)\expect{C(P_{tran}^*(t), S(t)) \left|\right.\bv{Z}(t)} \nonumber \\
&& + U(t) \expect{m(A(t), k^*(t))\left|\right. \bv{Z}(t)} \nonumber \\
&&  - X(t)\expect{d_{av} - d(A(t), k^*(t)) \left|\right.\bv{Z}(t)} \nonumber \\
&& + V\expect{P_{tran}^*(t) + \phi(A(t), k^*(t)) \left|\right.\bv{Z}(t)} \label{eq:appd}
\end{eqnarray}
where $k^*(t) \in \script{K}$ and $P_{tran}^*(t) \in \script{P}$ are any other feasible
control actions for slot $t$.   We obtain bounds on $\overline{U}$, $\overline{P}_{tot}$, 
and $\overline{D}$ using three different $k^*(t)$ and $P_{tran}^*(t)$ policies.

\begin{itemize} 
\item ($\overline{U}$ Analysis): Let $P_{tran}^*(t) = P_{max}$, and let 
$k^*(t)$ be the stationary randomized policy that makes decisions independently
of the current queue state, and yields the minimum output rate $r_{d, min}$ from
the compressor, subject to the distortion constraint: 
\begin{eqnarray*}
\expect{m(A(t), k^*(t))} &=& r_{d, min} \\
\expect{d(A(t), k^*(t))} &\leq& d_{av} 
\end{eqnarray*}
Plugging this into (\ref{eq:appd}) yields: 
\begin{eqnarray*}
&& \hspace{-.3in} \Delta(\bv{Z}(t)) + V\expect{P_{tot}(t) \left|\right. \bv{Z}(t) } \leq \\
&& C - U(t)[r_{max} - r_{d, min}] + V[P_{max} + \phi_{max}]
\end{eqnarray*}
and hence: 
\begin{eqnarray*}
&& \hspace{-.35in} \Delta(\bv{Z}(t)) \leq 
C - U(t)[r_{max} - r_{d, min}] + V[P_{max} + \phi_{max}]
\end{eqnarray*}
Using this drift inequality directly in the Lyapunov Drift Lemma (Lemma \ref{lem:lyap-drift})
yields the bound on $\overline{U}$ given in (\ref{eq:u-dist-bound}). 

\item ($\overline{P}$ Analysis): Let $P_{tran}^*(t)$ and $k^*(t)$ be the stationary randomized
algorithms that choose actions independently of queue backlog and yield: 
\begin{eqnarray*}
\expect{C(P_{tran}^*(t), S(t))} &=& r^* \\
\expect{P_{tran}^*(t)} &=& g^*(r^*) \\
\expect{m(A(t), k^*(t))} &=& r^* \\
\expect{\phi(A(t), k^*(t))} &=& h_d^*(r^*)  \\
 \expect{d(A(t), k^*(t))} &\leq& d_{av} 
\end{eqnarray*}
where $r^*$ is the optimal solution of problem (\ref{eq:problem-dist}), satisfying: 
\[ h_d^*(r^*) + g^*(r^*) = P_{av}^* \]
Plugging this into (\ref{eq:appd}) yields: 
\begin{eqnarray*}
 \Delta(\bv{Z}(t)) + V\expect{P_{tot}(t) \left|\right. \bv{Z}(t) } \leq 
 C + VP_{av}^*
\end{eqnarray*}
Using this drift inequality in the Lyapunov Drift Lemma (Lemma \ref{lem:lyap-drift}) 
yields the $\overline{P}_{tot}$ bound of (\ref{eq:p-dist-bound}). 

\item ($\overline{D}$ Analysis): Let $P_{tran}^*(t) = P_{max}$ and 
let $k^*(t)$ be any stationary randomized policy that is independent
of queue backlog and that yields: 
\begin{eqnarray*}
\expect{m(A(t), k^*(t))} &\leq& r_{max} \\
\expect{d(A(t), k^*(t))} &=& d_{av} - \epsilon
\end{eqnarray*}
for some value $\epsilon>0$. 
Plugging this into (\ref{eq:appd}) yields: 
\begin{eqnarray*}
&& \hspace{-.3in} \Delta(\bv{Z}(t)) + V\expect{P_{tot}(t) \left|\right. \bv{Z}(t) } \leq \\
&& C - X(t)\epsilon + V[P_{max} + \phi_{max}]
\end{eqnarray*}
and hence: 
\[  \Delta(\bv{Z}(t)) \leq  C - X(t)\epsilon + V[P_{max} + \phi_{max}] \]
Using this drift inequality directly  in the Lyapunov Drift Lemma (Lemma 
\ref{lem:lyap-drift}) 
yields: 
\[ \limsup_{t\rightarrow\infty}\frac{1}{t}\sum_{\tau=0}^{t-1} \expect{X(\tau)} \leq \frac{C + V[P_{max} + \phi_{max}]}{\epsilon} \]
It follows that the virtual queue $X(t)$ is strongly stable.  Because it has a finite maximum
departure rate $d_{av}$, the time average expected arrival rate to $X(t)$
(given by $\overline{D}$) is less than or equal to the time average expected transmission
rate (given by $d_{av}$) \cite{neely-energy-it} \cite{now}.  This proves (\ref{eq:d-dist-bound}). 
\end{itemize} 

\section*{Appendix E -- Derivation of $U_{thresh}$ for the Logarithmic Rate-Power Curve  Model}
The logarithmic model has $C(P) = \alpha \log(1 + \beta P)$ (using a natural log), with 
$\mu_{max} = \alpha\log(1 + \beta P_{max})$.  The dynamic transmission algorithm
solves: 
\begin{eqnarray}
\mbox{Maximize:} & U(t)\alpha\log(1 + \beta P) - VP  \label{eq:maximize} \\
\mbox{Subject to:} & 0 \leq P \leq P_{max} \nonumber
\end{eqnarray}
The largest value of  $U_{thresh}$ that satisfies Property 1 is given by $U_{thresh} = \max[0, \theta]$, 
where $\theta$ is the  minimum value for the following optimization problem: 
\begin{eqnarray}
\mbox{Minimize:} & \theta = U - \mu^*(U) \label{eq:uthresh-log-1} \\
\mbox{Subject to:} & \frac{V}{\alpha\beta} \leq U \leq \frac{V}{\alpha\beta} + \frac{VP_{max}}{\alpha} \nonumber
\end{eqnarray}
where $\mu^*(U) = \alpha\log(1 + \beta P^*(U))$ and $P^*(U)$ is the optimum solution to 
(\ref{eq:maximize}) for $U(t) = U$, given by: 
\[ P^*(U) = \left[\frac{U\alpha}{V} - \frac{1}{\beta}\right]_0^{P_{max}} \]
where the operator $[x]_0^{y}$ is defined $ [x]_0^{y} \defequiv \max[0, \min[x, y]]$.  This can be understood as follows: 
The value $U - \mu^*(U)$ is the queue backlog after transmission when $U(t) = U$.  If 
$U(t) \leq V/(\alpha\beta)$ then $P(t) = 0$ and $\mu(t) = 0$ (so queue backlog cannot further 
decrease)
while if $U(t) \geq V/(\alpha\beta) + VP_{max}/\alpha$ then $P(t) = P_{max}$ and $\mu(t) = \mu_{max}$, so the queue cannot drop below the value it would drop to  if $U(t) = 
V/(\alpha\beta) + VP_{max}/\alpha$. 

The problem (\ref{eq:uthresh-log-1}) reduces to: 
\begin{eqnarray*}
\mbox{Minimize:} & U - \alpha\log(U\alpha\beta/V) \\
\mbox{Subject to:} & \frac{V}{\alpha\beta} \leq U \leq \frac{V}{\alpha\beta} + \frac{VP_{max}}{\alpha} 
\end{eqnarray*}
The critical points of the above problem appear at the two endpoints of the interval
and at the point $U= \alpha$ (if this point is inside the interval).  
We thus have: 
\begin{equation} 
U_{thresh} = \left\{\begin{array}{ll} 
\max[0, \alpha - \alpha \log(\alpha^2\beta/V)] & \mbox{ if $\frac{V}{\alpha\beta} \leq \alpha \leq \frac{V}{\alpha\beta} + \frac{VP_{max}}{\alpha}$} \\
\max\left[0, \min\left[\frac{V}{\alpha\beta}, \frac{V}{\alpha\beta} + \frac{VP_{max}}{\alpha} - \mu_{max}\right]\right] & \mbox{ else} 
                            \end{array}
                                 \right. \label{eq:uthresh-log}
\end{equation}

\bibliographystyle{acmtrans}
\bibliography{../../../../../latex-mit/bibliography/refs}
\end{document}